%% file: qOU_EVT3.tex
\documentclass[10pt]{amsart}
\def\cal#1{\mathcal{#1}}

          \usepackage{amssymb}
          \usepackage{amsmath}
          \usepackage{amsthm}
          \usepackage{enumerate}
         \usepackage{url}
          \usepackage{color}

\usepackage[numbers,sort&compress]{natbib}
\input{include_macros2014}

\usepackage{epsfig}
\usepackage{ifthen}

\newtheorem{Thm}{Theorem}[section]
\newtheorem{Lem}[Thm]{Lemma}
\newtheorem{Prop}[Thm]{Proposition}

\theoremstyle{definition}
\newtheorem{Rem}[Thm]{Remark}

\numberwithin{equation}{section}

\usepackage{color,soul}
\renewcommand{\comment}[1]{{\color{blue}\fbox{#1}}}

\newcommand{\longcommenthide}[1]{}

\title{Extremes of $q$-Ornstein--Uhlenbeck processes}
\author{Yizao Wang}
\address
{
Yizao Wang\\
Department of Mathematical Sciences\\
University of Cincinnati\\
2815 Commons Way\\
Cincinnati, OH, 45221-0025, USA.
}
\email{yizao.wang@uc.edu}
\begin{document}\sloppy

\date{\today}


\begin{abstract}
Two limit theorems are established on the extremes of  a family of stationary Markov processes, known as $q$-Ornstein--Uhlenbeck processes with $q\in(-1,1)$.  Both results are crucially based on the weak convergence of the tangent process at the lower boundary of the domain of the process,  a positive self-similar Markov process little investigated so far in the literature. The first result is the asymptotic excursion probability established by the double-sum method, with an explicit formula for the Pickands constant in this context. The second result is a Brown--Resnick-type limit theorem on the minimum process of i.i.d.~copies of the $q$-Ornstein--Uhlenbeck process: with appropriate scalings in both time and magnitude, a new semi-min-stable process arises in the limit. \end{abstract}
\maketitle
\section{Introduction}
In this paper, we continue our investigation 
 on the path properties of  $q$-Ornstein--Uhlenbeck processes ($q\in(-1,1)$) in \citep{bryc16local,wang16large}. These are stationary Markov processes with explicit  transition probability density functions. This family of processes have two origins. On one hand, as Markov processes they, more precisely certain transformations of them called $q$-Brownian motions, arise as a special case of the  quadratic harnesses introduced in \citep{bryc05conditional,bryc07quadratic}. In short, quadratic harnesses are centered and square-integrable  stochastic processes $\{X_t\}_{t\in[0,\infty)}$ such that $\esp(X_sX_t) = \min(s,t)$ and for $0\le r<s<t$ given the past $\{X_u\}_{u\le r}$ and future $\{X_u\}_{u\ge t}$, the conditional mean and variances of $X_s$ are in linear and quadratic forms of $(X_r, X_t)$, respectively. On the other hand, the $q$-Ornstein--Uhlenbeck processes and $q$-Brownian motions arise for the first time in non-commutative probability with the same name, and it is known since the seminal results of~\citet{biane98processes} and \citet{bozejko97qGaussian} that every non-commutative Markov process has a classical Markov process counterpart that we investigate here. In this paper, we shall focus on $q$-Ornstein--Uhlenbeck processes as classical Markov processes. No knowledge of non-commutative probability is needed. 

As the name tells, the $q$-Ornstein--Uhlenbeck process has an intriguing connection to the well investigated Ornstein--Uhlenbeck Gaussian process: as $q\uparrow1$, the former converges weakly to the latter. This makes one wonder to what extend the two processes are similar. For each $q$ fixed, however, we have seen that in terms of path properties, the $q$-Ornstein--Uhlenbeck processes are qualitatively different \citep{bryc16local,wang16large}. For example, each $q$-Ornstein--Uhlenbeck process has bounded state space $[-2/\sqrt{1-q},2/\sqrt{1-q}]$, while the Ornstein--Uhlenbeck process takes values in $\R$. Moreover, we know now that the $q$-Ornstein--Uhlenbeck processes have jumps, and more precisely  for all $q\in(-1,1)$ they behave locally as Cauchy processes: this is established via the framework of tangent processes \citep{falconer03local} in \citep{bryc16local}; at the same time, the Ornstein--Uhlenbeck process has continuous sample paths. 

In this paper, we study the extremes of $q$-Ornstein--Uhlenbeck processes, as a continuation of our previous investigations.   Extreme value theory for stochastic processes has been extensively developed. There are already excellent monographs on both the general theory~\citep{leadbetter83extremes,piterbarg96asymptotic,resnick87extreme,dehaan06extreme} and concrete examples from a broad range of applications  \citep{aldous89probability}, just to mention a few. Here, we address two important problems on the extremes of continuous-time stochastic processes, in the example of $q$-Ornstein--Uhlenbeck processes. The first is to establish the asymptotic excursion probability of the  process over a fixed interval. The second is to establish the weak convergence of the maximum process of i.i.d.~copies of the same process. 

There exists already a huge literature on the two problems for Gaussian processes, and we take the same strategies to investigate the $q$-Ornstein--Uhlenbeck processes. 
What makes the analysis of extremes in this case special, however, is the tangent process of $q$-Ornstein--Uhlenbeck processes at the boundary of the domain, a positive self-similar Markov process that has not been much investigated so far in the literature. As a consequence, a new stationary process arises in the answer to the second question. 
Below, we first review related results for stationary Gaussian processes, to shed light on the techniques to be applied to the $q$-Ornstein-Uhlenbeck processes, and particularly on  the crucial role of the tangent process when answering both questions.
\subsection{Review of extremes of Gaussian processes}

Let $\G\topp\alpha\equiv \{\G\topp\alpha_t\}_{t\in\R}$ be a stationary centered Gaussian process, and assume 
\equh\label{eq:cov}
\cov(\G\topp\alpha_0,\G\topp\alpha_t) = 1-C|t|^\alpha + o(|t|^\alpha) \mbox{ for some $\alpha\in(0,2], C>0$ as $t\to 0$}.
\eque
We assume stationarity for the sake of simplicity. Many results are known for $\G\topp\alpha$ being non-stationary. Two questions of our interest here are the following.
\medskip

\noindent {(i) \em Asymptotic excursion probability of stationary Gaussian processes.} 
In a seminal work \citet{pickands69asymptotic} showed that
\equh\label{eq:excursion}
\proba\pp{\sup_{t\in[0,L]}\G\topp\alpha_t>u}\sim LC^{1/\alpha}H_\alpha\frac {1}{\sqrt {2\pi}}u^{2/\alpha-1}e^{-u^2/2}
\eque
as $u\to\infty$ where $H_\alpha$ is the so-called Pickands constant. The Pickands constant can be expressed by
\[
H_\alpha:=\lim_{T\to\infty}\frac{H_\alpha(T)}T \qmwith H_\alpha(T) = \int_0^\infty e^s\proba\pp{\sup_{t\in[0,T]}\pp{\sqrt 2\,\B^{\alpha/2}_t - t^\alpha}>s}ds,
\]
where $\B^{\alpha/2}$ denotes a standard fractional Brownian motion with Hurst index $\alpha/2\in(0,1]$, that is, a centered Gaussian process with 
\[
\cov(\B^{\alpha/2}_s,\B^{\alpha/2}_t) = \frac12\pp{s^\alpha+t^\alpha-|s-t|^\alpha}, s,t\ge 0.
\] 
\medskip

\noindent{(ii) \em Maximum process of i.i.d.~stationary Gaussian processes.} Let $\G\topp\alpha_i\equiv\{\G\topp\alpha_{i,t}\}_{t\in\R}$ be i.i.d.~copies of $\G\topp\alpha$. \citet{brown77extreme} and \citet{kabluchko09stationary} showed that, for 
\[
b_n := \sqrt{2\log n} - \frac{\frac12\log\log n + \log(2\sqrt\pi)}{\sqrt{2\log n}} \qmand  a_n := \frac1{C^{1/\alpha}b_n^{2/\alpha}},
\]
\equh\label{eq:BR}
b_n\ccbb{{\max_{i=1,\dots,n}\G\topp\alpha_{i,a_nt}-b_n}}_{t\in[0,\infty)} \weakto \ccbb{\zeta\topp\alpha_t}_{t\in[0,\infty)}
\eque
in $D([0,\infty))$, 
where $\{\zeta\topp\alpha_t\}_{\in[0,\infty)}$ is a {\em max-stable process} \citep{dehaan84spectral,kabluchko09spectral,gine90max,stoev06extremal,kabluchko09extremes}, now known as the Brown--Resnick process.
Here and in the sequel, we let ``$\weakto$'' denote weak convergence of probability measures \citep{ethier86markov,billingsley99convergence}.

One way to represent the limiting process $\zeta\topp\alpha$ is the following. Let $\indn U$ be enumerations of points from a Poisson point process on $\R$ with intensity measure $e^{-u}du$, and let $\indn{\B^{\alpha/2}}$ be i.i.d.~copies of a standard fractional Brownian motion, independent from $\indn U$. Then one can define $\zeta\topp\alpha$ via
\equh\label{eq:zeta}
\zeta\topp\alpha_t:=\sup_{n\in\N}\pp{U_n+\sqrt 2\,\B^{\alpha/2}_{n,t}-t^\alpha}, t\in[0,\infty).
\eque 
\citet{brown77extreme} actually worked out first for the case where $\G$ is an Ornstein--Uhlenbeck process, while the results by \citet{kabluchko09stationary} allow more general assumptions  than~\eqref{eq:cov}. Since then, the Brown--Resnick processes attracted much attention in the community of extreme value theory. For recent developments, see for example \citep{engelke15max,oesting12simulation,engelke11equivalent,engelke15maxima,dieker15exact}. \medskip

Now we comment on the strategies of the proofs of both results mentioned above. 
First, it is not a coincidence that fractional Brownian motions $\B^{\alpha/2}$ show up in both results. 
Indeed, the drifted fractional Brownian motions $\{\sqrt 2\,\B^{\alpha/2}_t-t^\alpha\}_{t\in[0,\infty)}$ arise as the {\em tangent process} of the Gaussian process $\G\topp\alpha$, conditioning on $\G_0\topp\alpha$ being increasingly large, under the assumption~\eqref{eq:cov}. More precisely, consider 
\[
\wt\G_t\topp {\alpha,n}:=b_n\pp{\G_{a_nt}\topp\alpha-b_n}, t\ge 0, n\in\N.
\]
It is an easy exercise to show that, under~\eqref{eq:cov} and under the law $\proba(\cdot\mid \wt\G\topp {\alpha,n}_{0} = w)$,
\equh\label{eq:tangent_G}
\ccbb{\wt \G_{t}\topp {\alpha,n} - w}_{t\in[0,\infty)}\weakto  \ccbb{\sqrt2\, \B^{\alpha/2}_t - {t^\alpha}}_{t\in[0,\infty)},
\eque
as $n\to\infty$, by computing the means and covariances. Note that as $n\to\infty$, $\wt\G\topp{\alpha,n}_0 = w$ says $\G\topp{\alpha}_0 = b_n+w/b_n\to\infty$. Therefore, we refer to the limit in~\eqref{eq:tangent_G} as the {\em tangent process at the boundary}, viewing the infinity as the boundary of the domain of $\G\topp\alpha$. Intuitively, the limiting process $\{\sqrt 2\,\B^{\alpha/2}_t-t^\alpha\}_{t\in[0,\infty)}$ explains the asymptotic behavior of $\G\topp\alpha$ right after achieving a high value: it drops down like a drifted fractional Brownian motion, after  appropriate normalization. 
This weak convergence plays an important role in both results described above. 

Besides,  when establishing~\eqref{eq:excursion}, Pickands developed a simple and yet powerful method, the so-called {\em double-sum method} in \citep{pickands69asymptotic}. This method has turned out to be successful in establishing asymptotic excursion probabilities for stochastic processes and random fields not necessarily stationary or Gaussian. See for example \citep{debicki16extremes,hashorva16extremes,albin90extremal,cheng16mean,cheng16excursion}, just to mention a few. 
When establishing~\eqref{eq:BR}, another useful tool, the {\em convergence of point processes}, is needed. The Poisson point process $\indn U$ in~\eqref{eq:zeta} is the limit of order statistics of $\{\wt G_{i,0}\topp {\alpha,n}\}_{i=1,\dots,n}$.  This tool has also been ubiquitous in the literature of extreme value theory. See for example \citep{lepage81convergence,resnick87extreme}. 

\subsection{Overview of main results}
We establish the counterparts of~\eqref{eq:excursion} and~\eqref{eq:BR} for $q$-Ornstein--Uhlenbeck processes, denoted by $\{X\topp q_t\}_{t\in\R}$ from now on.
Note that as in the Gaussian case, the processes are symmetric and hence working with the maximum/supremum is equivalent to work with minimum/infimum up to some sign changes. {\em Here, equivalently we work with infimum excursion and minimum process in both problems}. The only purpose of this change of convention is for the tangent process to have support on $[0,\infty)$ instead of $(-\infty,0]$. 

Recall that the $q$-Ornstein--Uhlenbeck processes have bounded domain $[b_q^-,b_q^+]$ with $b_q^\pm = \pm 2/\sqrt{1-q}$ for all $q\in(-1,1)$.
We will first show under the law $\proba(\cdot\mid X\topp q_0 = b_q^-+w\epsilon^2), w>0$,
\[
\sqrt{1-q}\cdot\ccbb{\frac{X\topp q_{\epsilon t}-b_q^-}{\epsilon^2}}_{t\in[0,\infty)}\weakto \ccbb{\Z_t^w}_{t\in[0,\infty)}
\]
in $D([0,\infty))$. 
The limit process $\Z^w$  is the tangent process at the lower boundary, and from now on is referred to as the tangent process for short. It is a positive self-similar Markov process, starting at $\Z^w_0 = w$. It is worth mentioning that  $\Z$ has its connection to a process in non-commutative probability via Biane's construction~\citep{biane98processes}, as explained in \citep{bryc16local}, but we do not need this fact here. In the world of classical probability, however, we do not know any other results on $\Z$ besides this limit theorem.
The finite-dimensional convergence was obtained in~\citep{bryc16local}. In Section~\ref{sec:tangent} we establish the tightness, by computing the convergence of semigroups of the corresponding Markov processes and applying a result from \citet[Theorem 4.2.11]{ethier86markov}. 

Once weak convergence to the tangent process is established, the asymptotic excursion probability can be computed by the double-sum method. Theorem~\ref{thm:doublesum} is the  counterpart of~\eqref{eq:excursion}, where we also provide an expression of the Pickands constant in this context. A technical issue is that at a few places we need an estimate of the probability of the type
\[
\proba\pp{\inf_{t\in[0,L]}\Z_t^w<1}
\]
as $w\to\infty$. For this  we need an inequality due to \citet{khoshnevisan97escape,xiao98asymptotic}.

The most interesting result is in Section~\ref{sec:BR}, where we establish another  Brown--Resnick-type limit theorem as~\eqref{eq:BR} in Theorem~\ref{thm:BR}, in the form of
\equh\label{eq:BRZ}
\ccbb{\min_{i=1,\dots,n} \frac{X_{i,a_n t}\topp{q}-b_q^-}{b_n}}_{t\in\R}\weakto\{\eta(t)\}_{t\in\R}
\eque
in $D(\R)$, where $\{X_{i,t}\topp q\}_{t\in\R}, i\in\N$ are i.i.d.~copies of the $q$-Ornstein--Uhlenbeck process.  The limit minimum process $\eta$ can be constructed as
\[
\eta(t) := \inf_{n\in\N}\Z_{n,t}^{W_n}, t\in\R,
\]
where $\{W_n\}_{n\in\N}$ are enumerations of points from  a Poisson point process on $\R_+$ with intensity $(3/2)w^{1/2}dw$, for each $n\in\N$  the process $\Z_n^{W_n}$ is a two-sided version of $\Z$ starting from $W_n$, and $\{\Z_n^{W_n}\}_{n\in\N}$ are conditionally independent given $\{W_n\}_{n\in\N}$ (see Section~\ref{sec:BR} for more details). 
Similarly to the process $\{U_n+\sqrt 2\B^{\alpha/2}_{n,t}-t^\alpha\}_{t\ge 0}$, the presentation in \eqref{eq:zeta}, the process $\{\Z^w_t\}_{t\ge 0}$ drifts to infinity as $t\to\infty$ (see Proposition \ref{prop:transience}). 
 This process provides a new example to the general framework of stationary systems of Markov processes investigated in \citet{brown70property,engelke15max}. 

The limit minimum process provides another rare example of a semi-min-stable process that is not min-stable. Observe that in~\eqref{eq:BR} and~\eqref{eq:BRZ}, the processes are scaled in both magnitude and time. In general when considering minimum of i.i.d.~copies of stochastic processes, it is well known that if the temporal  scaling is not  allowed, all non-degenerate limit processes that can arise are {\em min-stable} processes. If in addition the temporal scaling is allowed, \citet{penrose92semi} provided a characterization of all  possible limit processes under mild assumptions as the class of {\em semi-min-stable} (SMS) processes, which  contains min-stable processes as a subclass. At the same time, SMS processes are also {\em min-infinitely-divisible} (min-i.d.) processes. 
Notable references on min-stable and min-i.d.~processes, or their max counterparts max-stable and max-i.d.~processes, include \citep{dehaan84spectral,dehaan86stationary,gine90max,stoev06extremal,kabluchko09spectral,kabluchko16stochastic,weintraub91sample,stoev08ergodicity,kabluchko10ergodic,dombry13regular}, which provide a general framework to treat such processes by the so-called spectral representations, among other contributions. 
There exists already an extensive literature on such processes.

As for limit theorems as we considered here, however, very few results are known besides the aforementioned Brown--Resnick processes~\eqref{eq:BR} established in~\citep{kabluchko09stationary,brown77extreme}. \citet{engelke15max} established limit theorems for the minimum process of  i.i.d.~Ornstein--Uhlenbeck processes driven by skewed $\alpha$-stable noise, and showed that the limit processes belong to the class of so-called L\'evy--Brown--Resnick processes, a generalization of Brown--Resnick processes introduced in the same paper. However, it was also shown in \citep{engelke15max} that all L\'evy--Brown--Resnick processes are min-stable. The first example of limit minimum processes that are SMS but not min-stable  is due to 
\citet{penrose91minima}, who  examined the minimal distance to the origin of independent Brownian particles, and showed that the limit minimum process  is the infimum of countably infinite Bessel processes with  scalings and shifts following an independent Poisson point process. To the best of our knowledge, the limit minimum process of i.i.d.~$q$-Ornstein--Uhlenbeck processes is the second example of non-min-stable SMS processes that arise in a Brown--Resnick-type limit theorem.

The paper is organized as follows. In Section~\ref{sec:markov} we present preliminary results on $q$-Ornstein-Uhlenbeck processes as Markov processes. In Section~\ref{sec:tangent} we establish the weak convergence to the tangent process $\Z$. In Section~\ref{sec:EVT} we compute the asymptotic excursion probability. In Section~\ref{sec:BR} we describe the limit infimum process and establish the Brown--Resnick-type limit theorem. 
\section{Preliminaries on related Markov processes}\label{sec:markov}
We first recall the definition of $q$-Ornstein--Uhlenbeck processes, denoted by $X\topp q= \{X\topp q_t\}_{t\in\R}$.
The marginal distribution of the $q$-Ornstein--Uhlenbeck process is a symmetric probability measure supported on  $b_q^-\leq x\leq b_q^+$ with
\[
b_q^\pm := \pm \frac2{\sqrt{1-q}}, q\in(-1,1),
\] and has
  probability density function
 \[ 
f\topp q(x) := \frac{\sqrt{1-q}\cdot(q)_\infty}{2\pi}\sqrt{4-(1-q)x^2}\prodd k1\infty\bb{(1+q^k)^2-(1-q)x^2q^k} \indd{|x|\leq \frac2{\sqrt{1-q}}},
\]
where  $(q)_\infty := \prodd k1\infty(1-q^k)$.
This distribution is sometimes called the $q$-normal distribution and appears also  as the orthogonality measure of the   $q$-Hermite polynomials \citep[Section 13.1]{ismail09classical}.
It is known that $X\topp q$ is a stationary Markov process with {\em c\`adl\`ag} trajectories,  with the  transition probability density function $f\topp q_{s,t}(x,y)$
given by
\[
f\topp q_{s,t}(x,y) := (e^{-2(t-s)};q)_\infty \prodd k0\infty\frac1{\varphi_{q,k}(t-s,x,y)} \cdot f\topp q(y)\cdot \indd{|x|\leq \frac2{\sqrt{1-q}}},
\]
with
\[
\varphi_{q,k}(\delta,x,y) := (1-e^{-2\delta}q^{2k})^2 - (1-q)e^{-\delta}q^k(1+e^{-2\delta}q^{2k})xy + (1-q)e^{-2\delta}q^{2k}(x^2+y^2).
\]
Here and below, we write
\[
(a;q)_\infty := \prodd k0\infty(1-aq^k),\mfa  a\in\R, q\in(-1,1).
\]
The above densities can be found at \citep[Corollary 2]{bryc05probabilistic}. See also \citep{szablowski12qWiener} for more background.

Most time we shall work with a transformation of $X\topp q$, namely 
\[
\widetilde X\topp{q,\epsilon}_t \defe \sqrt{1-q}\cdot \frac{X\topp q_{\epsilon t}-b_q^-}{\epsilon^2}.
\]
We let 
 $p\topp{q,\epsilon}_{s,t}$ denote the transition density function of $\wt X\topp{q,\epsilon}$:
 \[
 p\topp{q,\epsilon}_{s,t}(x,y) := f_{\epsilon s,\epsilon t}\topp q\pp{b_q^-+\frac x{\sqrt{1-q}}\epsilon^2, b_q^-+\frac y{\sqrt{1-q}}\epsilon^2}\frac{\epsilon^2}{\sqrt{1-q}}, s,t,x,y>0,
 \]
and $p\topp{q,\epsilon}$ the marginal probability density function \[
p\topp{q,\epsilon}(x) := f\topp q\pp{b_q^-+\frac{x\epsilon^2}{\sqrt{1-q}}}\frac{\epsilon^2}{\sqrt{1-q}}, x\ge 0.
\]

Another Markov process we will work with is the tangent process $\Z$ which takes values in $(0,\infty)$ and has transition density function \equh\label{eq:Biane1/2}
p_{s,t}(x,y) :=
\frac{2 \left(t-s\right)  \sqrt{y}}{\pi
  \bb{(y-x)^2+2 (x+y)(t-s)^2+(t-s)^4}}\indd{x,y>0},s<t.
\eque
The transformed process $\wt X\topp{q,\epsilon}$ is convenient to work with, since we have shown in \citep{bryc16local} that
 \[
 \lim_{\epsilon\downarrow0}p_{s,t}\topp{q,\epsilon}(x,y) = p_{s,t}(x,y) \mfa 0\leq s<t,x,y>0.
 \]
We shall strengthen this result to the convergence of the semigroups in Section~\ref{sec:tangent}, which leads to the weak convergence of $\wt X\topp{q,\epsilon}$ to $\Z^w$ under the law $\proba(\cdot\mid \wt X_0\topp{q,\epsilon}=w)$. 
Throughout, we denote the transition functions of the Markov processes $\wt X\topp{q,\epsilon}$ and $\Z$ respectively by
\[
P\topp{q,\epsilon}_t(x,A) := \int_Ap\topp{q,\epsilon}_{0,t}(x,y)dy \qmand P_t(x,A) := \int_Ap_{0,t}(x,y)dy
\]
for all $t>0, A\in\calB(\R_+)$, and the corresponding semigroups by 
 \equh\label{eq:semigroup_X}
P_t \topp {q,\epsilon}f(x) := \int_0^\infty p\topp{q,\epsilon}_{0,t}(x,y) f(y)dy, t>0, x\in[0,4/\epsilon^2],
 \eque
 and
  \equh\label{eq:semigroup_Z}
P_tf(x):=\int_0^\infty p_{0,t}(x,y)f(y)dy \equiv \int_0^\infty\frac{2t\sqrt yf(y)}{\pi[t^4+2t^2(x+y)+(x-y)^2]}dy, t> 0, x>0.
 \eque
Strictly speaking the semigroups are not defined on the same space of functions, but this will cause little inconvenience when proving convergence, as explained in Theorem~\ref{thm:4.2.11}. For now, in~\eqref{eq:semigroup_X} and~\eqref{eq:semigroup_Z} it suffices to consider $f\in B([0,\infty))$, the space of bounded and measurable functions on $[0,\infty)$.
 
It is shown in \citep{szablowski12qWiener} that the semigroup of $X\topp q$ is Feller. Since $\wt X\topp{q,\epsilon}$ is a linear transformation of $X\topp q$,  $\{P\topp{q,\epsilon}_t\}_{t\ge 0}$ is Feller. Here we show that   $\{P_t\}_{t\ge 0}$ is also Feller. Let $\what C(\R)$ denote the space of all continuous functions on $\R_+:=[0,\infty)$ such that $\lim_{x\to\infty} f(x) = 0$, equipped with the sup norm $\nn f_\infty:=\sup_{x\in\R_+}|f(x)|$. 
\begin{Lem}\label{lem:continuity}
The semigroup $(P_t)_{t\ge0}$ is a Feller semigroup in the sense that for all $f\in \what C(\R_+)$, $P_tf\in \what C(\R_+)$ for all $t>0$ and $\lim_{t\downarrow0}\nn{P_tf-f}_{\infty} = 0$. 
\end{Lem} 
\begin{proof}
From~\eqref{eq:semigroup_Z}, the fact that $P_tf$ is continuous follows from the dominated convergence theorem. To see $\lim_{x\to\infty}|P_tf(x)| = 0$, for all $\epsilon>0$ choose $M_\epsilon>0$ such that $\sup_{y>M_\epsilon}|f(y)|<\epsilon$, and observe
\begin{multline*}
\limsup_{x\to\infty}|P_tf(x)|\leq \limsup_{x\to\infty}\int_0^{M_\epsilon}p_t(x,y)|f(y)|dy + \epsilon \\
\leq\limsup_{x\to\infty}\int_0^{M_\epsilon}\frac{\sqrt y\,{ \nn f_\infty}}{\pi xt}dy+\epsilon = \epsilon.
\end{multline*} 
So $P_tf\in \what C(\R_+)$. For the second statement, by \citep[Proposition III.2.4]{revuz99continuous} it suffices to show
\[
\lim_{t\downarrow0}P_tf(x) = f(x), \mfa x\in\R_+.
\]
To see this, we have
\begin{multline*}
|P_tf(x) - f(x)| \\\leq  \int_{y>0, |x-y|\le\delta}p_t(x,y)|f(y)-f(x)|dy + \int_{y>0, |x-y|>\delta}p_t(x,y)|f(y)-f(x)|dy\\
\leq 
 \sup_{y>0,|y-x|\le \delta}|f(y)-f(x)| + \frac{4t\nn f_\infty}\pi \int_{y>0,|x-y|>\delta}\frac{\sqrt y}{(x-y)^2}dy.
\end{multline*}
On the right-hand side of the last inequality above, the first term can be arbitrarily small by taking $\delta>0$ small enough due to the continuity of $f$ at $x$, and the second term  goes to zero as $t\downarrow0$. 
\end{proof}

Throughout, we use a generic symbol $\proba$ to denote the laws of different Markov processes, for the sake of simplicity. These processes are not necessarily defined on a common probability space, but we always assume that they take values  in the space $D$. Moreover, when indicating the law of a Markov process, either $\wt X\topp{q,\epsilon}$ or $\Z$, starting from a fixed point $w$ at time zero, we use the notation $\proba(\cdot\mid \wt X\topp{q,\epsilon}_0 = w)$ or $\proba(\Z^w\in\cdot)$, respectively. 


An important property of the tangent process is self-similarity. That is,
\equh\label{eq:SS}
\ccbb{\Z_{\lambda t}^w}_{t\ge 0}\eqd \lambda^2\ccbb{\Z^{w/\lambda^2}_t}_{t\ge0}, \mfa w,\lambda>0.
\eque
Here and in the sequel, we let `$\eqd$' denote `equal in finite-dimensional distributions'. 

It is also useful to keep in mind that the Markov process $\Z$ has stationary distribution with density proportional to $w^{1/2}dw$. This is easy to see as, for $\pi(w) = w^{1/2}$, we have $\pi(x)p_{s,t}(x,y) = \pi(y)p_{s,t}(y,x)$ for all $s<t$ and $x,y>0$. So the stationary distribution is infinite. Another useful fact is that  the process $\Z$ is transient.
\begin{Prop}\label{prop:transience}
For all $w\ge 0$, $\lim_{t\to\infty}\Z_t^w = \infty$ almost surely.
\end{Prop}
\begin{proof}
We follow \citep[p.~89, (4.24)]{blumenthal68markov}. By \citep[p.~89, (4.23)]{blumenthal68markov}, $\Z$ is a standard process. So it suffices to verify the two assumptions in (4.24) therein. 

First, introduce $Uf(x) = \int_0^\infty P_tf(x)dt$. We show that for all $f\in\what C(\R_+)$, $Uf$ is continuous. Then, 
\equh\label{eq:Blumenthal}
|Uf(x)-Uf(x')| \le \int_0^\delta \abs{P_tf(x)-P_tf(x')}dt + \abs{\int_\delta^\infty P_tf(x)-P_tf(x')dt}.
\eque
For the first term on the right-hand side above, it can be bounded by 
\begin{multline*}
\int_0^\delta|P_tf(x) - f(x)|+|f(x)-f(x')|+|f(x')-P_t(x')|dt \\
\le 
\delta\pp{2\sup_{t\in[0,\delta]}\nn{P_tf-f}_\infty+|f(x)-f(x')|}.
\end{multline*}
For the second term on the right-hand side of \eqref{eq:Blumenthal}, for $\delta>0$ fixed, using the formula of $P_tf(x)$ in \eqref{eq:semigroup_Z} one can show 
$\lim_{x'\to x}\int_\delta^\infty P_tf(x')dt = \int_\delta^\infty P_tf(x)dt$, by the dominated convergence theorem. It then follows that $\lim_{x'\to x}|Uf(x)-Uf(x')|\le 2\delta \sup_{t\in[0,\delta]}\nn{P_tf-f}_\infty$. Letting $\delta\downarrow 0$, it follows from Lemma \ref{lem:continuity} that $Uf$ is continuous. 

Second, introduce $U(w,B) = \int_0^\infty \proba(\Z_t^w\in B)dt$. We show that for all $B = [0,K], K<\infty$, $w\ge 0$, $U(w,B)<\infty$. Indeed,
\begin{align*}
U(w,[0,K]) & = \int_0^1\proba(\Z_t^w\le K )dt + \int_1^\infty \proba(\Z_t^w\le K)dt \\
& \le 1 + \int_1^\infty\int_0^K \frac{2t\sqrt y}{\pi[(y-w)^2+2(y+w)t^2+t^4]}dydt\\
& \le 1+\frac 2\pi\int_1^\infty \frac1{t^3}dt \int_0^K\sqrt ydy<\infty. 
\end{align*}
We have thus verified that $\Z^w$ satisfies the two conditions  in \citep[p.~89, (4.24)]{blumenthal68markov}.
\end{proof}

\section{Weak convergence to the tangent process}\label{sec:tangent}

In this section we prove the following weak convergence of the tangent process.
 \begin{Thm}\label{thm:tangent}
       For all $q\in(-1,1)$, $w\ge 0$, under $\proba(\cdot\mid \wt X_0\topp {q,\epsilon} = w)$,
   \[
  \ccbb{\wt X\topp{q,\epsilon}_t}_{t\ge0}
  \weakto\ccbb{\Z_{t}^w}_{t\ge0}
   \]
in $D([0,\infty))$ as $\epsilon\downarrow 0$.
 \end{Thm}

Now to prove Theorem~\ref{thm:tangent}, we recall the following version of \citet[Theorem 4.2.11]{ethier86markov} that  characterizes the weak convergence of Markov processes by the corresponding semigroups.
\begin{Thm}\label{thm:4.2.11}
For the convergence of Theorem~\ref{thm:tangent} to hold, it suffices to show, for all $f\in\what C(\R_+)$,
\equh\label{eq:ethier}
\lim_{\epsilon\downarrow0}
\sup_{x\in[0,4/\epsilon^2]}\abs{P_t\topp {q,\epsilon} f(x) - P_tf(x)} = 0, \mfa t>0.
\eque
\end{Thm}
\begin{proof}
Fix $q$. For each $\epsilon>0$, $(P\topp{q,\epsilon}_t)_{t\ge 0}$ is a Feller semigroup on $B([0,4/\epsilon^2])$, the Banach space of bounded real-valued measurable functions on $[0,4/\epsilon^2]$ with supremum norm. 
At the same time, we have seen that $(P_t)_{t\ge0}$ is a Feller semigroup on $\what C(\R_+)\subset B(\R_+)$. The semigroups of interest, however,  are not defined on the same spaces. To deal with this issue, as in \citep[Theorem 4.2.11]{ethier86markov}, introduce $\pi_\epsilon:B(\R_+)\to B([0,4/\epsilon^2])$ defined by  $(\pi_\epsilon f)(x):= f(x)$ for all $f\in B(\R_+), x\in[0,4/\epsilon^2]$. Then, \citep[Theorem 4.2.11]{ethier86markov} states that the desired convergence follows from
\[
\lim_{\epsilon\downarrow0}\abs{P_t\topp{q,\epsilon}(\pi_\epsilon f)(x) - \pi_\epsilon (P_tf)(x)} = 0,
\]
which is equivalent to~\eqref{eq:ethier},
and the convergence of the initial distribution. The latter convergence is obvious.
\end{proof}

We prepare a few lemmas to start with. For convenience, write $p\topp{q,\epsilon}_{0,t}(x,y) = 0$ for all $x>4/\epsilon^2$.

\begin{Lem}\label{lem:pst_upperbound}
For all $q\in(-1,1)$, there exists a constant $C$ depending only on $q$, such that
\equh\label{eq:upper_pstq}
p_{0,t}\topp{q,\epsilon}(x,y) \leq C\frac{te^{2\epsilon t}\cdot \sqrt y}{16\sinh^4(\epsilon t/2)/\epsilon^4+(x-y)^2} \qmfa x,y,t,\epsilon>0.
\eque
\end{Lem}
\begin{proof}
Write
\begin{multline}\label{eq:3terms}
p_{0,t}\topp{q,\epsilon}(x,y) = \frac{\epsilon^2 (e^{-2\epsilon t};q)_\infty(q)_\infty}{2\pi}\\
\times \frac{\sqrt{4-\spp{2-y\epsilon^2}^2}\indd{y\in[0,4/\epsilon^2]}}{\varphi_{q,0}\spp{\epsilon t,b_q^-+\frac{x\epsilon^2}{\sqrt{1-q}},b_q^-+\frac{y\epsilon^2}{\sqrt{1-q}}}}\times \prodd k1\infty \frac{\psi_{q,k}(b_q^-+\frac{y\epsilon^2}{\sqrt{1-q}})}{\varphi_{q,k}(\epsilon t,b_q^-+\frac{x\epsilon^2}{\sqrt{1-q}}, b_q^-+\frac{y\epsilon^2}{\sqrt{1-q}})},
\end{multline}
with $\psi_{q,k}(x):= (1+q^k)^2-(1-q)x^2q^k$.

The first term is bounded by 
\[
\frac{\epsilon^2 (e^{-2\epsilon t};q)_\infty(q)_\infty}{2\pi}\leq \frac{\epsilon^3t\cdot (q)_\infty}\pi.
\]
For the second, since (see \citep[Section 2.1]{bryc16local})
\begin{multline*}
\varphi_{q,0}(\delta,x,y) 
= e^{-2\delta}\bb{4\sinh^2(\delta) + (1-q)(x-y)^2 + 2(1-q)xy(1-\cosh(\delta))}\\
\ge  e^{-2\delta}\bb{16\sinh^4(\delta/2) + (1-q)(x-y)^2},
\end{multline*}
we have
\[
 \frac{\sqrt{4-\spp{2-y\epsilon^2}^2}\indd{y\in[0,4/\epsilon^2]}}{\varphi_{q,0}\spp{\epsilon t,b_q^-+\frac{x\epsilon^2}{\sqrt{1-q}},b_q^-+\frac{y\epsilon^2}{\sqrt{1-q}}}}\leq \frac{2\sqrt y\epsilon}{e^{-2\epsilon t}[16\sinh^4(\epsilon t/2)+(x-y)^2\epsilon^4]}.
\]

For the third term, since 
\[
\prodd k1\infty\psi(x,y)\cdot \indd{|y|\leq 2/\sqrt{1-q}}\leq \prodd k1\infty (1+|q|^k)^2,
\]
and
\[
\min_{|x|,|y|\leq \frac2{\sqrt{1-q}}}\varphi_{q,k}(\delta,x,y) = (1-e^{-\delta}q^k)^4\geq (1-|q|^k)^4, k\in\N, \delta>0,
\]
(see \citep[Section 2.1]{bryc16local}), we have
\[
\prodd k1\infty\frac{\psi_{q,k}(y)}{\varphi_{q,k}(\delta,x,y)}\indd{|y|\leq 2/\sqrt{1-q}}\leq \prodd k1\infty\frac{(1+|q|^k)^2}{(1-|q|^k)^4}<\infty.
\]
The desired inequality now follows.
\end{proof}
In the sequel, for sequences of real numbers $\{a_\epsilon\}_{\epsilon>0}$ and $\{b_\epsilon\}_{\epsilon>0}$, we let $a_\epsilon\sim b_\epsilon$ as $\epsilon\downarrow0$ denote the asymptotic equivalence $\lim_{\epsilon\downarrow0}a_\epsilon/b_\epsilon = 1$. 

\begin{Lem}\label{lem:uniform}
Suppose $\prodd k1\infty a_k(\epsilon)$ and $\prodd k1\infty b_k(\epsilon)$ are absolutely convergent. If there exists a function $\gamma(\epsilon)$ such that 
\[
\sif k1 \abs{\frac{a_k(\epsilon)}{b_k(\epsilon)}-1}\leq \gamma(\epsilon) \qmand \lim_{\epsilon\downarrow0}\gamma(\epsilon) = 0,
\]
then there exists $\epsilon_0>0$ such that for all $\epsilon\in(0,\epsilon_0)$, 
\[
\abs{\prodd k1\infty\frac{a_k(\epsilon)}{b_k(\epsilon)}-1} \leq 4\gamma(\epsilon). 
\]
\end{Lem}
\begin{proof}
Indeed, it suffices to consider $\delta\in(0,1)$ such that for all $x\in[-\delta,\delta]$, $|\log(1+x)|\leq 2|x|$ and $|e^{x}-1|\leq 2|x|$. For such a $\delta$, let $\epsilon_0$ be small enough such that for all  $\epsilon\in(0,\epsilon_0)$, $\gamma(\epsilon)<\delta/2$. Then, for all $\epsilon\in(0,\epsilon_0)$, 
\[
\abs{\sif k1\log\frac{a_k(\epsilon)}{b_k(\epsilon)}} \leq 2\sif k1\abs{\frac{a_k(\epsilon)}{b_k(\epsilon)}-1} \leq 2\gamma(\epsilon)\leq \delta, 
\]
and thus
\[
\abs{\prodd k1\infty \frac{a_k(\epsilon)}{b_k(\epsilon)}-1} = \abs{\exp\pp{\sif k1 \log\frac{a_k(\epsilon)}{b_k(\epsilon)}}-1} \leq 4\gamma(\epsilon). 
\]
\end{proof}
\begin{Lem}\label{lem:uniform_convergence}
For all $q\in(-1,1)$, $T_1,T_2\in(0,\infty)$ with $T_1<T_2$, 
\[
p_{0,t}\topp{q,\epsilon}(x,y)\sim p_{0,t}(x,y) \mbox{ uniformly for all } x\in[0,M_1], y\in[0,M_2], t\in[T_1,T_2]
\]
as $\epsilon\downarrow0$, 
with the convention $0/0 = 1$. 

\end{Lem}

\begin{proof}
Recall~\eqref{eq:3terms}.
From now on, assume $x,y\in[0,M_2]$. The convergence of the first term on the right-hand side of~\eqref{eq:3terms} does not depend on $x$ nor $y$. For the second term, the numerator
$\sqrt{4-(2-y\epsilon^2)^2} \sim 2\epsilon\sqrt y$
as $\epsilon\downarrow0$, and the asymptotic equivalence is uniform for $y\in[0,M_2]$ (recall the convention $0/0 = 1$).  The denominator can be expressed as
\begin{multline*}
\varphi_{q,0}\pp{\epsilon t,b_q^-+\frac{x\epsilon^2}{\sqrt{1-q}}, b_q^-+\frac{y\epsilon^2}{\sqrt{1-q}}} 
= e^{-2\epsilon t}
\times \bigg[16\sinh^4\pp{\frac{\epsilon t}2} \\
+ (x-y)^2\epsilon^4 - 4(x+y)\epsilon^2(1-\cosh(\epsilon t)) + 2xy\epsilon ^4(1-\cosh (\epsilon t))\bigg],
\end{multline*}
which is asymptotically equivalent to 
\[
\epsilon^4\pp{t^4+(x-y)^2+2(x+y)t^2}
\]
as $\epsilon\downarrow 0$, uniformly for $x,y\in[0,M_2]$, $t\in[T_1,T_2]$. We have thus shown that the first two terms in~\eqref{eq:3terms} converges uniformly to $p_{0,t}(x,y)\cdot (q)_\infty^2$.

Next, we show the infinite product in~\eqref{eq:3terms} converges uniformly to 
\[
\prodd k1\infty\frac1{(1-q^k)^2} = \frac1{(q)_\infty^2}.
\]
For this purpose, we show
\equh\label{eq:infinite1}
\lim_{\epsilon\downarrow0}\prodd k1\infty \psi_{q,k}\pp{b_q^-+\frac{x\epsilon^2}{\sqrt{1-q}}} = \prodd k1\infty(1-q^k)^2
\eque
and
\equh\label{eq:infinite2}
\lim_{\epsilon\downarrow0}{\varphi_{q,k}\pp{\epsilon t,b_q^-+\frac{x\epsilon^2}{\sqrt{1-q}}, b_q^-+\frac{y\epsilon^2}{\sqrt{1-q}}}} = {(1-q^k)^4},
\eque
both uniformly for all $x,y\in[0,M_2]$, $t\in[0,T_2]$. Uniform convergence~\eqref{eq:infinite1}  follows from Lemma~\ref{lem:uniform} and the identity
\[
\abs{\frac{\psi_{q,k}(b_q^-+\frac{x\epsilon^2}{\sqrt{1-q}})}{ (1-q^k)^2}-1} =  \frac{q^k}{(1-q^k)^2}\abs{4x\epsilon^2-x^2\epsilon^4}.
\]
For the uniform convergence~\eqref{eq:infinite2},  consider
\[
\wt\varphi_{q,k}(\epsilon t):= (1-e^{-2\epsilon t}q^{2k})^2 - 4e^{-t\epsilon} q^k(1+e^{-2t\epsilon} q^{2k}) + 8 e^{-2\epsilon t}q^{2k}, k\in\N,
\]
and write, omitting the arguments for the sake of simplicity,
\equh\label{eq:infinite3}
\left.\prodd k1\infty\frac1{\varphi_{q,k}}\middle/\prodd k1\infty\frac1{(1-q^k)^4}\right. = \pp{\prodd k1\infty\frac1{\varphi_{q,k}}\middle/\prodd k1\infty\frac1{\wt\varphi_{q,k}}}  \cdot \pp{\prodd k1\infty\frac1{
\wt\varphi_{q,k}}/\prodd k1\infty\frac1{(1-q^k)^4}}.
\eque
To deal with the first term on the right-hand side of~\eqref{eq:infinite3}, one can show that there exists a constant $C$ such that
for all $x,y\in[0,M_2]$, 
\[
\bigg|\varphi_{q,k}\bigg(\epsilon t,b_q^-+\frac{x\epsilon^2}{\sqrt{1-q}},  b_q^-+\frac{y\epsilon^2}{\sqrt{1-q}}\bigg)  - \wt\varphi_{q,k}(\epsilon t)\bigg|
\leq Cq^k\pp{M_2\epsilon^2+M_2^2\epsilon^4}.
\]
So Lemma~\ref{lem:uniform} tells that
the first term on the right-hand side of~\eqref{eq:infinite3} tends to one uniformly.  For the second term, 
observe that
\[
\lim_{\epsilon\downarrow0}\wt\varphi_{q,k}(\epsilon t)= (1-q^{2k})^2 - 4q^k(1+q^{2k}) + 8 q^{2k} = (1-q^k)^4, k\in\N,
\]
and one can show similarly as above that
\[
\lim_{\epsilon\downarrow0} \prodd k1\infty\frac1{\wt\varphi_{q,k}(\epsilon t)}  = \prodd k1\infty \frac1{(1-q^k)^4}
\]
uniformly for $t\in[0,T_2]$. The proof is completed.
\end{proof}

  \begin{proof}[Proof of Theorem~\ref{thm:tangent}]
It suffices to prove~\eqref{eq:ethier}. Consider two constants $M_1, M_2>0$ to be determined later. Then,
\begin{multline*}
\sup_{x\in[0,4/\epsilon^2]}\abs{P_t\topp {q,\epsilon}f(x) - P_t\topp qf(x)}\\ \leq \sup_{x\in[0,4/\epsilon^2]}\abs{\int_0^{M_1}\pp{p_{0,t}\topp{q,\epsilon}(x,y) - p_{0,t}(x,y)}f(y)dy} + 2\sup_{y\ge M_1}|f(y)|.
\end{multline*}
For $f\in \what C(\R_+)$, the second term on the right-hand side above is arbitrarily small by taking $M_1$ sufficiently large. For the first term on the right-hand side above, it can be bounded from above by, for $\epsilon<(4/M_2)^{1/2}$,
\begin{multline*}
\nn f_\infty\sup_{x\in[0,M_2]}\int_0^{M_1}\abs{p_{0,t}\topp{q,\epsilon}(x,y) - p_{0,t}(x,y)}dy
\\
+ \nn f_\infty\sup_{x\in(M_2,4/\epsilon^2]}\bb{P\topp{q,\epsilon}_t(x,[0,M_1]) + P_t(x,[0,M_1])}.
\end{multline*}
\medskip

\noindent (i) For 
$P_t\topp{q,\epsilon}(x,[0,M_1]) = \int_{0}^{M_1} p_{0, t}\topp {q,\epsilon}(x,y)dy$, 
by Lemma~\ref{lem:pst_upperbound}, there exists a constant $C$ depending only on $q$ and $t$, such that
\[
p_{0,t}\topp {q,\epsilon}\pp{x,y}\leq C\frac{\sqrt{M_1}}{(M_2-M_1)^2} \mfa \epsilon>0,y\in[0,M_1],x\in(M_2,4/\epsilon^2].
\]
So
\[
\sup_{x\in(M_2,4/\epsilon^2]}P_t\topp{q,\epsilon}(x,[0,M_1]) \leq C\frac{M_1^{3/2}}{(M_2-M_1)^2}.\medskip
\]
(ii) For $\sup_{x\in(M_2,4/\epsilon^2]}P_t(x,[0,M_1])$, it is bounded from above by $2tM_1^{3/2}/(\pi(M_2-M_1)^2)$, by recalling~\eqref{eq:Biane1/2}.\medskip

\noindent (iii) Next, we prove
\equh\label{eq:uniform}
\lim_{\epsilon\downarrow0}\sup_{x\in[0,M_2]}\int_0^{M_1}\abs{p_{0,t}\topp{q,\epsilon}(x,y)-p_{0,t}(x,y)}dy = 0.
\eque
Since when $\epsilon$ is small enough,
$\sup_{x\in[0,M_2],y\in[0,M_1]}p_{0,t}(x,y) = 2\sqrt{M_1}/(\pi t^3)<\infty$, now~\eqref{eq:uniform} follows from Lemma~\ref{lem:uniform_convergence}. 
 To sum up, we have shown that, there exists a constant $C$ depending only on $q$ and $t$, such that
\[
\limsup_{\epsilon\downarrow0}\sup_{x\in[0,4/\epsilon^2]}\abs{P_t\topp {q,\epsilon}f(x) - P_tf(x)}\leq \frac{CM_1^{3/2}}{(M_2-M_1)^2}\nn f_\infty + \sup_{y>M_1}|f(y)|.
\]
Taking $M_2 = 2M_1$ and $M_1$ arbitrarily large, the desired result follows.
\end{proof}
\section{Asymptotic excursion probability}\label{sec:EVT}
The goal of this section is to establish the asymptotic excursion probability by Pickands' double-sum method \citep{pickands69asymptotic,piterbarg96asymptotic}. 

To define the so-called Pickands constant in this case, we first define
\[
H(T):={ \int_0^\infty\sqrt w\cdot \proba\pp{\inf_{t\in[0,T]}\Z_t^w<1}dw},
\]
where $\Z_t^w$ is the Markov process with transition density function~\eqref{eq:Biane1/2} starting from $w$. We first show that $H(T)<\infty$. 
For this purpose, we need the following lemma due to \citet{khoshnevisan97escape,xiao98asymptotic}.
\begin{Lem}\label{lem:KX}
Let $\{X_t\}_{t\ge 0}$ be a strong Markov process on $[0,\infty)$ with transition probability $P_t^X$. Then, for all constants $S,T$ such that $0\leq S<T<\infty$,
\[
\proba\pp{\inf_{t\in[S,T]}X^a_t\leq x} \leq \frac{\int_S^{2T-S}P_t^X(a,[0,x])dt}{\inf_{y\in[0,x]}\int_0^{T-S}P_t^X(y,[0,x])dt}.
\]
\end{Lem}
\begin{proof}
Set $\tau_x^a:=\inf\{t\ge S: X^a_t\leq x\}$. Then,
\begin{align*}
\int_S^{2T-S}P_t^X(a,[0,x])dt & = \esp \pp{\int_S^{2T-S}\indd{X_t\in[0,x]}dt}\\
&  \ge \esp \pp{\int_S^{2T-S}\indd{X_t\in[0,x]}dt\cdot \indd{\tau_x\le T}} 
\\
& \ge \esp \bb{\indd{\tau_x\le T}\esp\pp{\int_{\tau_x}^{\tau_x+T-S}\indd{X_t\in[0,x]}dt\mmid \tau_x}}\\
& \ge \proba(\tau_x\le T)\inf_{y\in[0,x]}\int_0^{T-S}P_t^X(y,[0,x])dt.
\end{align*}
\end{proof}
By this lemma, 
\[
\proba\pp{\inf_{t\in[0,T]}\Z_t^w\leq 1} \leq \frac{\int_0^{2T}P_t(w,[0,1])dt}{\inf_{y\in[0,1]}\int_0^TP_t(y,[0,1])dt}.
\]
By the density formula of $p_{s,t}$ in~\eqref{eq:Biane1/2}, for $w\ge 2$, the numerator is bounded from above by $8T/(\pi w^2)$.  For any $\delta\in(0,T\wedge 1)$, the denominator is bounded from below by
\equh\label{eq:infp}
\inf_{y\in[0,1]}\int_\delta^TP_t(y,[\delta,1])dt \ge (T-\delta)(1-\delta)\inf_{y\in[0,1],z\in[\delta,1],t\in[\delta,T]}p_{0,t}(y,z)>0.
\eque So we have shown that there exists a constant $C$ such that
\equh\label{eq:infZ}
\proba\pp{\inf_{t\in[0,T]}\Z_t^w\leq 1} \leq \frac C{w^2}\wedge 1, \mfa w\ge 0,
\eque
whence $H(T)<\infty$ for all $T>0$.

The main result of this section is the following. 
\begin{Thm}\label{thm:doublesum}
For all $q\in(-1,1)$, $L>0$,
\[
u_L\topp {q,\epsilon}:= \proba\pp{\inf_{t\in[0,L]}X_t\topp q<b_q^-+\frac{\epsilon^2}{\sqrt{1-q}}}
\sim \epsilon^2\frac{(q)_\infty^3}{\pi} \cdot LH
\]
as $\epsilon\downarrow 0$,
where
\equh\label{eq:H}
H := \lim_{T\to\infty}\frac{H(T)}T
\eque
is a well-defined, strictly positive and finite constant. 
\end{Thm}
The constant $H$ is the so-called Pickands constant in this case. 

Throughout we fix $L>0$. It is convenient to work with $\wt X\topp{q,\epsilon}$. So we write
\[
u_L\topp {q,\epsilon} \equiv  \proba\pp{\inf_{t\in[0,L]}\wt X_t\topp{q,\epsilon}<1},
\]
and introduce
\begin{multline*}
A_{i,T}\topp {q,\epsilon}:=\ccbb{\inf_{t\in((i-1)T\epsilon,iT\epsilon]}X_t\topp q<b_q^-+\frac{\epsilon^2}{\sqrt{1-q}}}
\equiv\ccbb{\inf_{t\in((i-1)T,iT]}\wt X_{t}\topp{q,\epsilon}<1} 
\end{multline*}
for $\epsilon,T>0, i\in\N$. 
The idea of the double-sum method is to observe, for $N_{T,\epsilon} = \sfloor{L/(T\epsilon)}$,
\equh
\label{eq:doublesum}
\summ i1{N_{T,\epsilon}}\proba\pp{A_{i,T}\topp {q,\epsilon}}
- \sum_{i=1}^{N_{T,\epsilon}}\sum_{\substack{j=1\\j\neq i}}^{N_{T,\epsilon}}\proba\pp{A_{i,T}\topp {q,\epsilon}\cap A_{j,T}\topp {q,\epsilon}} \leq u_L\topp {q,\epsilon}\leq \summ i1{N_{T,\epsilon}+1}\proba\pp{A_{i,T}\topp {q,\epsilon}}.
\eque
Now, we start with two lemmas on the limits of the two summands above. 
\begin{Lem}\label{lem:doublesum1}For all $q\in(-1,1)$, $T>0$,
\[
\proba\pp{A_{1,T}\topp {q,\epsilon}} \sim    \epsilon^3\cdot \frac{(q)_\infty^3}\pi\int_0^\infty \sqrt w\cdot \proba\pp{\inf_{t\in[0,T]}\Z_t^w<1}dw
\]
as $\epsilon\downarrow0$.
\end{Lem}
\begin{proof}
By the Markov property, we write
\[
\proba\pp{A_{1,T}\topp {q,\epsilon}}  
 = \int_0^\infty p\topp {q,\epsilon}(w) \proba\pp{A_{1,T}\topp {q,\epsilon}\mmid \wt X_0\topp {q,\epsilon} = w}dw,
\]
keeping in mind that the integrand is zero for $w>4/\epsilon^2$. 
Introduce
\[
s\topp q_{T,\epsilon}(w):= \proba\pp{A_{1,T}\topp {q,\epsilon}\mmid \wt X_0\topp{q,\epsilon} = w}.
\]
The goal is to show
\[
\int_0^{\infty}p\topp {q,\epsilon}(w)s_{T,\epsilon}\topp q(w)dw  \sim\epsilon ^3\frac{(q)_\infty^3}\pi\int_0^\infty \sqrt w\cdot \proba\pp{\inf_{t\in[0,T]}\Z_t^w<1}dw,
\]
as $\epsilon\downarrow0$,
by applying the dominated convergence theorem. The pointwise convergence is straightforward: we have
\begin{multline*}
p\topp {q,\epsilon}(w) = \epsilon^2\frac{(q)_\infty}{2\pi}\sqrt{4-(2-w\epsilon^2)^2}\prodd k1\infty \psi_{q,k}\pp{b_q^-+\frac{w\epsilon^2}{\sqrt{1-q}}}\cdot \indd{w\in[0,4/\epsilon^2]}\\
\sim \frac{(q)_\infty^3}{\pi}\sqrt w\epsilon^3,
\end{multline*}
and 
\[
\lim_{\epsilon\downarrow0}
s\topp {q,\epsilon}_{T}(w)= 
\proba\pp{\inf_{t\in[0,T]}\Z^w_t<1}
\]
by  Theorem~\ref{thm:tangent} and the continuous mapping theorem. 

We now find an integrable upper bound for $p\topp {q,\epsilon}s\topp{q,\epsilon}_T$.
For $p\topp {q,\epsilon}$, observe that for some constant $C$,
\equh\label{eq:upper_pq}
p\topp {q,\epsilon}(w)\leq \epsilon^2\frac{(q)_\infty}{2\pi}\sqrt{4-(2-w\epsilon^2)^2}\prodd k1\infty (1+|q|^k)^2\cdot\indd{w\in[0,4/\epsilon^2]}\leq C \sqrt w\epsilon^3.
\eque
For an upper bound of $s\topp{q,\epsilon}_T$, observe that 
by Lemma~\ref{lem:KX}
\equh\label{eq:KX}
s\topp{q,\epsilon}_T(w)\leq \frac{\int_0^{2T}P_t\topp {q,\epsilon}(w,[0,1])dt}{\inf_{y\in[0,1]}\int_0^T P_t\topp{q,\epsilon}(y,[0,1])dt}.
\eque
We shall derive an upper bound of $s\topp{q,\epsilon}_T(w)$ from here for $w\ge 2$ first. For the numerator of the right-hand side of~\eqref{eq:KX},
we have for some constant $C$, 
\[
\int_0^{2T}P_t\topp{q,\epsilon}(w,[0,1])dt\leq 2T \sup_{\substack{y\in[0,1]\\t\in[0,2T]}}p_{0,t}\topp{q,\epsilon}(w,y)\leq \frac C{w^2},\mfa w\ge 2,
\]
where the last inequality follows from~\eqref{eq:upper_pstq}.

For a lower bound of the denominator on the right-hand side of~\eqref{eq:KX}, we use, for some $\delta\in(0,(T\wedge1))$,
\begin{multline*}
\inf_{y\in[0,1]}\int_0^T P_t\topp{q,\epsilon}(y,[0,1])dt \ge 
\inf_{y\in[0,1]}\int_\delta^T P_t\topp{q,\epsilon}(y,[\delta,1])dt \\
\ge
(T-\delta)(1-\delta) \inf_{y\in[0,1],z\in[\delta,1], t\in[\delta,T]} p_{0,t}\topp{q,\epsilon}(y,z).
\end{multline*}
The last term above is strictly positive for $\epsilon$ small enough, again by Lemma~\ref{lem:uniform_convergence} and~\eqref{eq:infp}. So we have shown 
\equh\label{eq:sTq}
s_T\topp{q,\epsilon}(w)\leq \frac C{w^2}\wedge1 \mfa w>0.
\eque

To sum up, we have shown that there exists constant $C$ such that for $\epsilon$ small enough, 
\equh\label{eq:bound_rqsq}
p\topp {q,\epsilon}(w)s_{T}\topp {q,\epsilon}(w)\leq C\epsilon^3\pp{\frac 1{w^{3/2}}\wedge 1}, \mfa w\ge 0.
\eque
Therefore the dominated convergence theorem yields the desired result.
\end{proof}

\begin{Lem}
For all $q\in(-1,1)$, $T>0$,
\begin{align*}
\lim_{\epsilon\downarrow0} & \frac1{\epsilon^2}\sum_{i=1}^{N_{T,\epsilon}}\sum_{\substack{j=1\\j\neq i}}^{N_{T,\epsilon}} \proba\pp{A_{i,T}\topp {q,\epsilon}\cap A_{j,T}\topp {q,\epsilon}}\\
& 
= L\cdot\frac{(q)_\infty^3}{\pi T}\int_0^\infty \sqrt w\cdot \proba\pp{\inf_{t\in[0,T]}\Z_t^w<1}{\sif i1\proba\pp{\inf_{t\in[iT,(i+1)T]}\Z_t^w<1}}dw \\
& \le L\cdot\frac{H(T)^2}{T^4} \cdot\frac{2(q)_\infty^3}{\pi^2} \sif i1\frac1{i^3}.
\end{align*}
\end{Lem}

\begin{proof}
We  start by showing that the summation of infinite probabilities is finite. For this, by the Markov property,
\begin{multline*}
\proba\pp{\inf_{t\in[iT,(i+1)T]}\Z_t^w<1} = \int_0^\infty p_{0,iT}(w,y)\proba\pp{\inf_{t\in[0,T]}\Z_t^y<1}dy \\
\leq \frac2{\pi (iT)^3}\int_0^\infty\sqrt y\cdot \proba\pp{\inf_{t\in[0,T]}\Z_t^y<1}dy = \frac{2H(T)}{\pi i^3T^3},
\end{multline*}
where in the inequality we applied the inequality that $p_{0,t}(w,y)\le 2\sqrt{y}/(\pi t^3)$ (see~\eqref{eq:upper_pstq}).
Thus,
\[
\sif i1\proba\pp{\inf_{t\in[iT,(i+1)T]}\Z_t^w<1}
\leq \frac{2H(T)}{\pi T^3} \sif i1 \frac1{i^3}<\infty.
\]

Now for each $i\ge 1$, notice that the two events $A_{1,T}\topp{q,\epsilon}$ and $A_{i+1,T}\topp{q,\epsilon}$ are determined by $\wt X\topp {q,\epsilon}_{t\in[0,T]}$ and $\wt X\topp {q,\epsilon}_{t\in[iT,(i+1)T]}$ respectively. 
Conditioning on the position of $\wt X_T\topp {q,\epsilon} = w$, the two processes are independent. 
More specifically, view the latter 
as the Markov process with the same semigroup starting at $w$ over time interval $[(i-1)T,iT]$, and the former as the reversed Markov process $\what X_t\topp{q,\epsilon} := \wt X_{T-t}\topp{q,\epsilon}, t\in[0,T]$, again starting at $\what X_0\topp {q,\epsilon} = w$. 
Since the original $q$-Ornstein--Uhlenbeck process is stationary, so is $\wt X\topp{q,\epsilon}$. It then follows that $\what X\topp{q,\epsilon}$ has the same semigroup as $\wt  X\topp{q,\epsilon}$. In particular,
\begin{multline*}
\proba\pp{A_{1,T}\topp{q,\epsilon}\mmid\wt X\topp{q,\epsilon}_T = w} = \proba\pp{\inf_{t\in[0,T]}\what X_t\topp{q,\epsilon}\mmid \what X_0 \topp{q,\epsilon} = w} \\
= \proba\pp{\inf_{t\in[0,T]}\wt X_t\topp{q,\epsilon}\mmid \wt X_0 \topp{q,\epsilon} = w} = \proba\pp{A_{1,T}\topp{q,\epsilon}\mmid\wt X\topp{q,\epsilon}_0 = w}, \mfa w>0.
\end{multline*}  Therefore, we have 
\begin{multline*}
\proba\pp{A_{1,T}\topp {q,\epsilon}\cap A_{i+1,T}\topp {q,\epsilon}} 
\\
 = \int_0^\infty p\topp {q,\epsilon}(x)\proba\pp{A_{1,T}\topp {q,\epsilon}\mmid \wt X_0\topp {q,\epsilon} = w}
\proba\pp{A_{i,T}\topp {q,\epsilon}\mmid \wt X_0\topp {q,\epsilon} = w}dw.
\end{multline*}
Similarly as in Lemma~\ref{lem:doublesum1}, we show by the dominated convergence theorem that
\begin{multline}\label{eq:doublesum1}
\proba\pp{A_{1,T}\topp {q,\epsilon}\cap A_{i,T}\topp {q,\epsilon}} 
\\
\sim \frac{\epsilon^3(q)_\infty^3}{\pi}\int_0^\infty \sqrt w\cdot \proba\pp{\inf_{t\in[0,T]}\Z_t^w<1}\proba\pp{\inf_{t\in[(i-1)T,iT]}\Z_t^w<1}dw,
\end{multline}
and 
\begin{multline}\label{eq:doublesum2}
\summ i1{N_{T,\epsilon}}\sum_{\substack{j=1\\j\neq i}}^{N_{T,\epsilon}} \proba\pp{A_{i,T}\topp {q,\epsilon}\cap A_{j,T}\topp {q,\epsilon}}\\
 \sim L\cdot\frac{\epsilon^2(q)_\infty^3}{T\pi}\int_0^\infty\sqrt w \cdot \proba\pp{\inf_{t\in[0,T]}\Z_t^w<1}\bb{\sif i1\proba\pp{\inf_{t\in[iT,(i+1)T]}\Z_t^w<1}}dw.
\end{multline}
Indeed, the pointwise convergences are straightforward, and it remains to find integrable upper bounds.  In addition to $s_{T}\topp {q,\epsilon}$, introduce
\[
s_{i,T}\topp {q,\epsilon}(w):= \proba\pp{A_{i+1,T}\topp {q,\epsilon}\mmid \wt X_0\topp {q,\epsilon} = w}.
\]
We have, for some constant $C$ not depending on $i$, 
\begin{multline*}
s_{i,T}\topp{q,\epsilon}(w) = \int_0^\infty p_{0,iT}\topp{q,\epsilon}(w,y)\proba\pp{A_{1,T}\topp {q,\epsilon}\mmid \wt X_0\topp q = y}dy \\
\leq \frac {CiT}{\sinh^4(\epsilon iT/2)/\epsilon^4}\int_0^\infty\sqrt y s_T\topp{q,\epsilon}(y)dy,
\end{multline*}
where the last inequality we applied~\eqref{eq:upper_pstq}, and the  integral is finite because of~\eqref{eq:sTq}. Therefore, 
\equh\label{eq:i^3}
\limsup_{\epsilon\downarrow0}s_{i,T}\topp{q,\epsilon}(w)\leq \frac C{(iT)^{3}}\mfa w\in\R_+.
\eque
 This and~\eqref{eq:upper_pq} yield~\eqref{eq:doublesum1}. For~\eqref{eq:doublesum2}, 
express the double sum as
\begin{multline*}
\summ i1{N_{T,\epsilon}}\sum_{\substack{j=1\\j\neq i}}^{N_{T,\epsilon}} \proba\pp{A_{i,T}\topp {q,\epsilon}\cap A_{j,T}\topp {q,\epsilon}} 
\\
=N_{T,\epsilon}\sif i1\pp{1-\frac i{N_{T,\epsilon}}}_+\int_0^\infty p\topp {q,\epsilon}(w)s_{T}\topp {q,\epsilon}(w)s_{i,T}\topp {q,\epsilon}(w)dw.
\end{multline*}
Now~\eqref{eq:bound_rqsq} and~\eqref{eq:i^3} provide an integrable upper bound for the integrant above,  and thus yield~\eqref{eq:doublesum2}.

\end{proof}
\begin{proof}[Proof of Theorem~\ref{thm:doublesum}]
Applying the previous two lemmas to~\eqref{eq:doublesum}, we obtain, for all $S,T>0$, 
\begin{multline}\label{eq:ST}
L\cdot \frac{H(S)}S - L\cdot \frac{H(S)^2}{S^4}\cdot {\frac{2(q)_\infty^3}{\pi^2}\sif i1\frac1{i^3}}
\leq\liminf_{\epsilon\downarrow0}\frac{u_L\topp {q,\epsilon}}{\epsilon^2 (q)_\infty^3/\pi}\\
\leq
\limsup_{\epsilon\downarrow0}\frac{u_L\topp {q,\epsilon}}{\epsilon^2 (q)_\infty^3/\pi}
\leq L\cdot\frac{H(T)}T.
\end{multline}
Again by Lemma~\ref{lem:doublesum1}, we have
\begin{multline*}
\frac LT \frac{\epsilon^2 (q)_\infty^3}\pi H(T)\sim 
\summ i1{N_{T,\epsilon}}\proba\pp{A_{i,T}\topp {q,\epsilon}}
 \leq \summ i1{N_{T,\epsilon}}(\floor T+1)\proba\pp{A_{i,1}\topp {q,\epsilon}} \\
\sim \frac{\sfloor T+1}TN_{1,\epsilon}\proba\pp{A_{i,1}\topp {q,\epsilon}} \sim \frac{(\sfloor T+1)L}T \frac{\epsilon^2(q)_\infty^3}\pi H(1) \mmas \epsilon\downarrow0.
\end{multline*}
So we have $H(T)\leq (\floor T+1)H(1)$ for all $T>0$, whence $\limsup_{T\to\infty}H(T)/T\leq H(1)<\infty$. Next, the left-hand side of~\eqref{eq:ST} is bounded from below by
\[
\frac{H(S)}S\pp{1-\frac C{S^2}},
\]
which is strictly positive for $S$ large enough. Fix such an $S$, and taking the limit on the right-hand side of~\eqref{eq:ST}, it follows that $\liminf_{T\to\infty}H(T)/T>0$. Now, taking the limit on both sides of~\eqref{eq:ST}, we have that
\[
0\le \limsup_{T\to\infty}\frac{H(T)}T\leq\liminf_{T\to\infty}\frac{H(T)}T<\infty.
\] That is, $H$ in~\eqref{eq:H} is a well-defined finite constant, and we have seen that it is strictly positive. The proof is thus completed.
\end{proof}

\section{Minimum process}\label{sec:BR}
For each $n\in\N$, let $\vv X_n\topp q\equiv\{X\topp q_{n,t}\}_{t\in\R}$ be an independent copy of $X\topp q$. We consider the non-degenerate limit for the process
\[
\ccbb{\frac{\min_{i=1,\dots,n}X_{i,a_nt}\topp q - b_q^-}{b_n}}_{t\in\R}
\]
as $n\to\infty$ 
in the space $D(\R)$, for some $\indn a$ and $\indn b$ appropriately chosen.
We first describe the limit process denoted by $\eta \equiv\{\eta(t)\}_{t\in\R}$ below. This  is a stationary process based on a stationary system of Markov processes considered in \citep{brown70property,engelke15max}. It is also a new example of the so-called semi-min-stable (SMS) processes introduced in \citep{penrose92semi}.
\subsection{Representations of limit minimum process}
We provide two representations of the limit minimum process. 
Recall that we characterize the tangent process $\{\Z_t\}_{t\ge0}$ by its semigroup~\eqref{eq:semigroup_Z}. The initial distributions that we shall consider are all in the form of  a unit point mass at some point $w\in(0,\infty)$, denoted by $\delta_w$. 
For both representations, we need to consider the two-sided extension of $\Z^w$ as a process defined on $\R$, still denoted as $\Z^w$ as follows. Let $\Z^{w,\pm}\equiv \{\Z^{w,\pm}_t\}_{t\ge 0}$  be two independent Markov processes with the same semigroup $(P_t)_{t\ge 0}$ and initial distribution $\delta_w$.
We assume that $\Z^{w,+}$ is in $D([0,\infty))$, and $\Z^{w,-}$ is in the space of functions that are left-continuous with right limits. Set $\Z^w_t := \Z^{w,+}_t$ if $t\ge 0$ and $\Z^w_t:= \Z^{w,-}_t$ otherwise. In this way, $\Z^w$ is in $D(\R)$. 
 We refer to the so-defined $\Z^w$ as the two-sided tangent process defined on the real line  starting from $w$.

Let $\indn{W}$ be enumerations of points from a Poisson point process defined on $\R_+$ with intensity $(3/2)w^{1/2}dw$, and for each $n\in\N$, let $\Z_n^{W_n}$ be a two-sided tangent process on $\R$ starting from $W_n$. It is assumed that $\indn{\Z^{W_n}}$ are conditionally independent given $\indn {W}$, and in $D(\R)$. 
Equivalently,  $\{\Z_n^{W_n}\}_{n\in\N}$ can be viewed as a Poisson point process, and the aforementioned construction  is a special case of the general framework of stationary systems of two-sided Markov processes considered in \citep{brown70property,engelke15max}. 
Now we consider
\equh\label{eq:eta}
\eta(t):= \inf_{n\in\N}\Z_{n,t}^{W_n}, t\in\R.
\eque
\begin{Lem}
The process $\{\eta(t)\}_{t\in\R}$ is in $D(\R)$ almost surely. It is a stationary process with  marginal distribution
\[
\proba(\eta(t)\leq x) = 1-\exp\pp{-x^{3/2}}, x>0,
\]
and finite-dimensional distribution, for all $m\in\N, t_1,\dots,t_m\in\R, x_1,\dots,x_m>0$,
\[
\proba\pp{\eta(t_1)> x_1,\dots,\eta({t_m})> x_m} = \exp\ccbb{-\int_0^\infty \proba\pp{\min_{i=1,\dots,m}\frac{\Z^w_{t_i}}{x_i}\leq 1}\frac32w^{1/2}dw}.
\]
\end{Lem}
\begin{proof}
We first show the process is in $D(\R)$. For this purpose, it suffices to show when restricted to any finite interval $[0,L]$, the process is in $D([0,L])$. Indeed, 
consider the set of indices
\[
A_L:= \ccbb{n\in\N: \inf_{t\in[0,L]}\Z_{n,t}^{W_n}\le M_{L}} \qmwith M_L:=\sup_{t\in[0,L]}\Z_{1,t}^{W_1}.
\]
Then, 
\equh\label{eq:finite}
\eta(t) = \inf_{n\in A_L}\Z_{n,t}^{W_n}, \mfa t\in[0,L],
\eque
and  it suffices to show that $|A_L|$ is almost surely finite. By the property of Poisson point process, 
 $|A_L|-1$ is distributed as a Poisson random variable with parameter 
\[
\int_0^\infty w^{1/2}\proba\pp{\inf_{t\in[0,L]}\Z_{n,t}^w\le M_L\mmid \Z_{1,\cdot}^{W_1}}dw,
\]
which is finite almost surely because of~\eqref{eq:infZ}.
Now,~\eqref{eq:finite} tells that almost surely over any finite interval, the process $\eta$ is the pointwise minimum of a finite number of processes in $D$, and hence also in $D$. 

The expression of finite-dimensional distributions follows from the definition of Poisson point processes. 
From there to obtain the marginal distribution, observe that
\[
\proba\pp{\eta(t)> x} = \exp\ccbb{-\int_0^\infty\proba\pp{\Z_t^w\le x}\frac32w^{1/2}dw}, t,x>0.
\]
By self-similarity~\eqref{eq:SS}, the integration on the right-hand side above equals
\begin{multline*}
\int_0^\infty\proba\pp{\Z_{t/x^{1/2}}^{w/x}\le1}\frac32w^{1/2}dw = x^{3/2}\int_0^\infty\proba\pp{\Z_{t/x^{1/2}}^w\le 1}\frac 32w^{1/2}dw \\
= x^{3/2}\int_0^\infty\proba(\Z_0^w\leq 1)\frac32w^{1/2}dw = x^{3/2}\int_0^1\frac32w^{1/2}dw= x^{3/2},
\end{multline*}
where in the second equality we used the fact that the stationary distribution of the Markov process $\Z$ has density proportional to $w^{1/2}dw$. 
A similar argument shows the stationarity of the one-sided process by computing finite-dimensional distributions. For the stationarity of the two-sided process, it suffices to recall 
 the construction of $\{\Z^w_t\}_{t\in\R}$ by  using the dual Markov process in the reversed direction. See 
\citep[Theorem 2.1]{engelke15max} for more details. In particular, for
$\Pi:=\sif n1 \delta_{\Z_{n,\cdot}^{W_n}}$
as a Poisson point process on $D(\R)$, it is stationary in the sense that
\[
\sif n1 \delta_{\Z_{n,\cdot}^{W_n}} \eqd \sif n1 \delta_{\Z_{n,\cdot +h}^{W_n}}\mfa h\in\R.
\]
\end{proof}

Now, we provide another representation of $\eta$ as a SMS process introduced by \citet{penrose92semi}. This class of stochastic processes forms  a special class of the min-i.d.~process. Recall that a non-negative stochastic process $\{Z_t\}_{t\in\R}$ is said to be min-i.d., if for all $n\in\N$, there exists i.i.d.~stochastic processes $\{Z\topp n_{i}\}_{i=1,\dots,n}$ such that
\[
\ccbb{Z_t}_{t\in\R} \eqd \ccbb{\min_{i=1,\dots,n}Z_{i,t}\topp n}_{t\in\R}.
\]
Furthermore, $Z$ is said to be SMS with parameter $(\alpha,\beta)\in(0,\infty)\times[0,\infty)$, if for i.i.d.~copies $\{Z_i\}_{i\in\N}$ of $Z$,
\equh\label{eq:SMS}
\ccbb{Z_t}_{t\in\R}\eqd \ccbb{n^{1/\alpha}\min_{i=1,\dots,n}Z_{i,t/n^{\beta}}}_{t\in\R} \mfa n\in\N.
\eque
In particular, if $Z$ is $(\alpha,0)$-SMS, then it is also a min-stable process with marginal $\alpha$-Weibull distribution; or equivalently,  $1/Z$ is an $\alpha$-Fr\'echet max-stable process. (Here it is only a matter of convention that which type of extreme value distributions to choose;  most literatures are based on either Fr\'echet or Gumbel distributions.)
An important result is due to 
\citet{penrose92semi}, who proved that the class of SMS processes coincides with the all the limit minimum processes of i.i.d.~copies of stochastic processes with appropriate scalings  in both time and magnitude. It is well known that when no temporal scaling is allowed, then the limit process is necessarily min-stable, although allowing temporal scaling does not necessarily lead to a non-min-stable process (e.g.~the Brown--Resnick process). 

Surprisingly, we are aware of only one example in the literature, due to  \citet{penrose91minima},
where the limit process is SMS but not min-stable. Here, the process $\eta$ in~\eqref{eq:eta} provides another example of such type.
To see this, we derive  an equivalent spectral representation of $\eta$. 
Observe
\begin{multline*}
\int_0^\infty \proba\pp{\min_{i=1,\dots,n}\frac{\Z^w_{t_i}}{x_i}\le 1} \frac32w^{1/2}dw = \int_0^\infty \proba\pp{w\min_{i=1,\dots,m}\frac{\Z^1_{t_i/w^{1/2}}}{x_i}\le 1}dw \\
= \int_0^\infty\proba\pp{y^{2/3}\min_{i=1,\dots,m}\frac{\Z^1_{t_i/y^{1/3}}}{x_i}\le 1}dy,
\end{multline*}
where we applied the self-similarity property~\eqref{eq:SS} of $\Z$ and change-of-variable $y = w^{3/2}$. This yields
\equh\label{eq:penrose}
\ccbb{\eta(t)}_{t\in\R}\eqd \ccbb{\inf _{n\in\N}W_n^{2/3} \Z^1_{n,tW_n^{-1/3}}}_{t\in\R},
\eque
where $\{W_n\}_{n\in\N}$ are enumerations of points from a standard Poisson point process, and $\{\Z^1_n\}_{n\in\N}$ are i.i.d.~copies of $\Z^1$, independent from $\indn W$. 
The representation~\eqref{eq:penrose} was introduced for general SMS processes under mild assumptions in addition to~\eqref{eq:SMS} in \citep[Example 1 and Theorem 5]{penrose92semi} (therein, $\alpha$-SMS corresponds to $(1,\alpha)$-SMS in our notation).   
For completeness we show that $\eta$ is $(3/2,1/3)$-semi-min-stable directly. Let $\{\eta_j\}_{j=1,\dots,n}$ be i.i.d.~copies of $\eta$. Then,
\begin{align*}
\proba& \pp{n^{2/3}\min_{j=1,\dots,n}\eta_j(t_i/n^{1/3})>x_i, i=1,\dots,m} \\
& = \exp\ccbb{- n\int_0^\infty \proba\pp{n^{2/3}y^{2/3}\min_{i=1,\dots,m}\frac{\Z^1_{t_i/(n^{1/3}y^{1/3})}}{x_i}\le1}dy}\\
&
= \proba\pp{\min_{i=1,\dots,m}\frac{\eta(t_i)}{x_i}>1}.
\end{align*}

\subsection{A Brown--Resnick-type limit theorem}
Recall that $\indn{\vv X\topp{q}}$ denote i.i.d.~copies of the process $\vv X\topp q\equiv \{X_t\topp q\}_{t\in\R}$. We then write accordingly $\wt {\vv X}_n\topp{q,\epsilon}:=\{\wt X_{n,t}\topp{q,\epsilon}\}_{t\in\R}$ with 
\[
\wt X_{n,t}\topp{q,\epsilon}:=\sqrt{1-q}\cdot \frac{X_{n,\epsilon t}\topp{q}-b_q^-}{\epsilon^2}, t\in\R, n\in\N. 
\]
The main result of this section is the following theorem.
\begin{Thm}\label{thm:BR}
For all  $q\in(-1,1)$,
\equh\label{eq:DR}
\ccbb{\min_{i=1,\dots,n}\wt X_{i,t}\topp{q,\epsilon_n}}_{t\in\R}\weakto \ccbb{\eta(t)}_{t\in\R}
\eque
in $D(\R)$ with
\[
\epsilon_n:= \pp{\frac{3\pi}{2(q)^3_\infty}}^{1/3}\frac1{n^{1/3}}, n\in\N.
\]

\end{Thm}
To prove the weak convergence in~\eqref{eq:DR}, it suffices to prove it in $D([-L,L])$ for all $L>0$, or equivalently in $D([0,2L])$ by stationarity. From now on we fix $L>0$ and focus on weak convergence in $D([0,L])$. 
We first introduce some notations for the transformed $q$-Ornstein--Uhlenbeck processes. Set
\[
W_i\topp{q,n} := \wt X_{i,0}\topp {q,\epsilon_n}, i\in\N.
\]
Consider the order-statistics of $\{W_i\topp{q,n}\}_{i=1,\dots,n}$ denoted by
\[
W_{1:n}\topp{q,n}\leq W_{2:n}\topp{q,n}\leq\cdots\leq W_{n:n}\topp{q,n}.
\]
The event that  all the inequalities above are strict  has probability one, and we shall focus on this event in the rest of this section. For each $n$,  order accordingly $\{\wt {\vv X}_i\topp{q,\epsilon_n}\}_{i=1,\dots,n}$ into
$\sccbb{\wt {\vv X}_{i:n}\topp{q,\epsilon_n}}_{i=1,\dots,n}$.

\begin{Lem}\label{lem:each_path}
for each $i$ fixed,
\equh\label{eq:each_path}
\ccbb{
\wt X_{i:n,t}\topp{q,\epsilon_n}}_{t\in[0,L]}\weakto \ccbb{\Z^{W_i}_{i,t}}_{t\in[0,L]}
\eque
in $D([0,L])$ as $n\to\infty$. 
\end{Lem}
\begin{proof}
We proceed by proving the convergence of the initial distributions and the semigroups respectively  \citep[Theorem 4.2.5]{ethier86markov}.
Recall that  each $\wt {\vv X}_{i:n}\topp{q,\epsilon_n}$ is a Markov process, the law of which is determined by the semigroup  $(P\topp{q,\epsilon_n}_t)_{t\ge 0}$ and the initial distribution $\delta_{W_{i:n}\topp{q,n}}$. 
We have seen the convergence of the semigroup in~\eqref{eq:ethier} before. It remains to prove $W_{i:n}\topp{q,n}\weakto W_i$ for each fixed $i$, which is a consequence of the point process convergence of the order statistics 
\equh\label{eq:order}
\summ i1n\delta_{W_i\topp{q,n}}\weakto\sif n1\delta_{W_n}
\eque
in the space of point measures. See \citep[Chapter 3]{resnick87extreme} for more details. In particular, the weak convergence~\eqref{eq:order} is equivalent to, by~\citep[Proposition 3.21]{resnick87extreme},
 \[
n\proba\pp{W_1\topp{q,\epsilon_n}\le x} = n\int_0^xp\topp{q,\epsilon_n}(y)dy \sim n\frac{(q)_\infty^3\epsilon_n^3}\pi\int_0^x\sqrt ydy = x^{3/2} , x>0,
\]
as $n\to\infty$. 
\end{proof}
We now rewrite for each $t\in\R$, the left-hand side of~\eqref{eq:DR} as
\[
Y_n\topp {q}(t):=\min_{i=1,\dots,n}\wt X_{i,t}\topp{q,\epsilon_n} = \min_{i=1,\dots,n}\wt X_{i:n,t}\topp{q,\epsilon_n},
\]
and compare to the limiting process $\eta$ in~\eqref{eq:eta}. 
By Lemma~\ref{lem:each_path}, we see formally the convergence of $Y_n\topp q$ to $\eta$ in a term-by-term manner. 
It takes some effort to make the argument rigorous.
\begin{proof}[Proof of Theorem~\ref{thm:BR}]
We proceed by an approximation argument. For all $\kappa>0$, consider
\[
\eta_\kappa^{\le}(t):=\inf_{n\in\N:W_n\leq\kappa}{\Z_{n,t}^{W_n}},\quad \eta_\kappa^{>}(t):=\inf_{n\in\N:W_n>\kappa}{\Z_{n,t}^{W_n}},
\]
and similarly
\[
Y_{\kappa,n}^{(q),\le}(t):=\min_{\substack{i=1,\dots,n\\W_i\topp {q,n}\leq \kappa}}\wt X_{i,t}\topp{q,\epsilon_n},\quad Y_{\kappa,n}^{(q),>}(t):=\min_{\substack{i=1,\dots,n\\W_i\topp {q,n}>\kappa}}\wt X_{i,t}\topp{q,\epsilon_n} .
\]
So
\[
Y_n\topp q(t)=\min\ccbb{Y_{\kappa,n}^{(q),\le}(t),Y_{\kappa,n}^{(q),>}(t)}
\qmand\eta(t) = \min\ccbb{\eta_\kappa^\le(t),\eta_\kappa^>(t)},t\in\R.
\]
The desired limit theorem now follows from the following three statements:
\equh\label{eq:BR2}
\ccbb{Y_{\kappa,n}^{(q),\le}(t)}_{t\in[0,L]}\weakto \ccbb{\eta_\kappa^{\le}(t)}_{t\in[0,L]} 
\quad\mmas n\to\infty \mfa \kappa>0,
\eque
\equh\label{eq:BR3}
\lim_{\kappa\to\infty}\limsupn\proba\pp{Y\topp q_{n}(t)\neq Y^{(q),\le}_{\kappa,n}(t) \mbox{ for some } t\in[0,L]} = 0,
\eque
and\equh\label{eq:BR1}
\lim_{\kappa\to\infty}\proba\pp{\eta(t)\neq \eta_\kappa^{\le}(t) \mbox{ for some } t\in[0,L]} = 0.
\eque

For~\eqref{eq:BR2}, by construction, $Y_{\kappa,n}^{(q),\le}$ is the infimum of a random finite number of trajectories in $D([0,L])$, denoted by
\[
N_{\kappa,n}\topp q:=\summ i1n\indd{W_i\topp{q,n}\le \kappa}. 
\]
By the convergence of order statistics established in~\eqref{eq:order}, we know that 
\[
N_{\kappa,n}\topp q\weakto \sif i1\indd{W_i\leq \kappa}=:N_\kappa,
\]
which is also finite almost surely.
Therefore, for all $\delta,\kappa>0$ there exists $m_{\kappa,\delta}\in\N$ such that
\equh\label{eq:m}
\limn \proba\pp{N\topp q_{\kappa,n}>m_{\kappa,\delta}} = \proba\pp{N_\kappa>m_{\kappa,\delta}}<\delta.
\eque
Now, consider
\[
Y_{\kappa,\delta,n}^{(q),\le}(t):=\min_{\substack{i=1,\dots,m_{k,\delta}\\W_{i:n}\topp{q,n}\le \kappa}}\wt X_{i:n,t}\topp{q,\epsilon_n}\qmand
\eta^\le_{\kappa,\delta}(t) := \min_{\substack{i=1,\dots,m_{\kappa,\delta}\\ W_i\leq\kappa}}\Z_{i,t}^{W_i}, t\in[0,L].
\]
We show
\equh\label{eq:Ykd}
\ccbb{Y_{\kappa,\delta,n}^{(q),\le}(t)}_{t\in[0,L]}\weakto\ccbb{\eta^\le_{\kappa,\delta}(t)}_{t\in[0,L]}
\eque
by the continuous mapping theorem. For this purpose, write $m = m_{\kappa,\delta}$ and introduce 
\[
f(x_1,\dots,x_m)(t):=\min_{\substack{i=1,\dots,m\\x_i(0)\leq\kappa}}x_i(t), \mfa x_i\in D([0,L]), i=1,\dots,m.
\]
Then one can express
\[
Y_{\kappa,\delta,n}^{(q),\le}(t) = f\pp{\wt X_{1:n}\topp{q,\epsilon_n},\dots,\wt X_{m:n}\topp{q,\epsilon_n}}(t)\qmand \eta_\kappa^\le(t) = f\pp{\Z_1^{W_1},\dots,\Z_{m}^{W_m}}(t),
\]
each as the same functional $f$ on $m$ processes in $D([0,L])$. The convergence of each argument above is established in Lemma~\ref{lem:each_path}. Therefore, it remains to show that $f$ is continuous with probability one. More precisely,
 let 
\[
\calD := \{\vv x\equiv (x_1,\dots, x_m)\in (D([0,L]))^m: \mbox{ $f$ is not continous at $\vv x$}\}
\] denote the discontinuity points of $f$ on $(D([0,L]))^m$. 
Observe that 
\begin{multline}\label{eq:discontinuous}
\ccbb{(Z_1^{W_1},\dots,Z_m^{W_m})\in\calD} 
\subset \ccbb{\exists i\in\{1,\dots,m\}, W_i = \kappa}\\
 \cup{
 \ccbb{\exists t\in(0,L), i\neq j\in\{1,\dots,m\} \mbox{ s.t.~$\Z_i^{W_i}$ and $\Z_j^{W_j}$ both have jumps at $t$}}}.
\end{multline}
The first event on the right-hand side above has zero probability. For the second, consider the discontinuity points, equivalently the jumps, of $\Z_i^{W_i}$ over $[0,L]$. Let $\tau\topp i_1$ denote the position of the largest jump in absolute value, $\tau\topp i_2$ the position of the second largest, and so on. 
Now, for each pair $i\neq j$, consider
\[
B_{i,j}:= \bigcup_{n=1}^\infty\ccbb{\Z_j^{W_j} \mbox{ has a discontinuity point at $\tau\topp i_n$}}.
\]
Recall the fact that Feller processes do not have fixed discontinuity points. This, in addition to the conditional independence of the two Markov processes given $W_i$ and $W_j$, leads to
\[\proba\pp{\Z_j^{W_j} \mbox{ has a discontinuity point at $\tau_n\topp i$}} = 0,
\]
whence $\proba(B_{i,j}) = 0$. We have shown  that the second event on the right-hand side of~\eqref{eq:discontinuous} has probability zero. We have thus shown~\eqref{eq:Ykd}.

Now, let $d$ denote the Skorohod metric on $D([0,L])$. 
Observe that on the event $\{\max\{N_{\kappa,n}\topp q,N_\kappa\}\le m_\kappa\}$, $Y_{\kappa,\delta,n}^{(q),\le} = Y\topp q_{\kappa,n}$ and $\eta_{\kappa,\delta}^\le = \eta_\kappa^\le$. Therefore, 
for all $\gamma>0$ we have
\begin{multline*}
\proba\pp{d\pp{Y_{\kappa,n}^{(q),\le},\eta_\kappa^\le}>\gamma}
 \leq \proba\pp{d\pp{Y_{\kappa,n}^{(q),\le},Y_{\kappa,\delta,n}^{(q),\le}}>\gamma/3}  \\+ \proba\pp{d\pp{Y_{\kappa,\delta,n}^{(q),\le},\eta_{\kappa,\delta}^\le}>\gamma/3} + \proba\pp{d\pp{\eta_{\kappa,\delta}^\le,\eta_\kappa^\le}>\gamma/3} \\
 \leq \proba\pp{N_{\kappa,n}\topp q>m_{\kappa,\delta}} + \proba\pp{N_\kappa>m_{\kappa,\delta}} + \proba\pp{d\pp{Y_{\kappa,\delta,n}^{(q),\le},\eta_{\kappa,\delta}^\le}>\gamma/3}.
\end{multline*}
It follows from~\eqref{eq:m} and~\eqref{eq:Ykd} that
 $\limsupn \proba(d(Y_{\kappa,n}^{(q),\le},\eta_\kappa^\le)>\gamma)<2\delta$, for all $\gamma,\delta>0$. This implies~\eqref{eq:BR2}. 

Next we prove~\eqref{eq:BR3}. For this, it suffices to show
\[
\lim_{\kappa\to\infty}\limsupn\proba\pp{\sup_{t\in[0,L]}\wt X_{1:n,t}^{(q),\epsilon_n}<\inf_{t\in[0,L]}Y_{\kappa,n}^{(q),>}(t)} = 0.
\]
Again, by~\eqref{eq:each_path}, $\{\wt{\vv  X}_{1:n}^{(q,\epsilon_n)}\}_{n\in\N}$ is a tight sequence in $D([0,L])$. So for all $\delta>0$, there exists $b_\delta$ such that
\[
\proba\pp{\sup_{t\in[0,L]}\wt X_{1:n,t}\topp{q,\epsilon_n}>b_\delta}<\delta. 
\]
We now show
\equh\label{eq:BR5}
\lim_{\kappa\to\infty}\limsupn\proba\pp{\inf_{t\in[0,L]}Y_{\kappa,n}^{(q),>}(t)<b_\delta} = 0. 
\eque
Indeed,
\begin{multline*}
\proba\pp{\inf_{t\in[0,L]}Y_{\kappa,n}^{(q),>}(t)<b_\delta} \leq \summ i1n \proba\pp{\inf_{t\in[0,L]}\wt X_{i,t}\topp{q,\epsilon_n}<b_\delta, \wt X_{i,0}\topp{q,\epsilon_n}>\kappa}\\
= n \proba\pp{\inf_{t\in[0,L]}\wt X_t\topp{q,\epsilon_n}<b_\delta, \wt X_0\topp{q,\epsilon_n}>\kappa} \\
= n\int_\kappa^\infty p\topp{q,\epsilon_n}(w)\proba\pp{\inf_{t\in[0,L]}\wt X_t\topp{q,\epsilon_n}<b_\delta\mmid \wt X_0\topp {q,\epsilon_n}=w}dw.
\end{multline*}
As we have seen in Lemma~\ref{lem:doublesum1} before, as $n\to\infty$ the last term above is asymptotically equivalent to 
\[
\frac32\int_\kappa^\infty \sqrt w\cdot \proba\pp{\inf_{t\in[0,L]}\Z_t^w<b_\delta}dw,
\]
which is finite for $\kappa = 0$. This  implies~\eqref{eq:BR5} and hence~\eqref{eq:BR3}.

It remains to show~\eqref{eq:BR1}. This follows from the  pointwise convergence of $\eta_\kappa^\le$ to $\eta$ over any finite interval, thanks to the presentation~\eqref{eq:finite} before. 
\end{proof}

\subsection{Asymptotic tail independence of the limit minimum process}\label{sec:tails}
As a process that arises in the investigation of extremes, the tail dependence of $\eta$ is of natural interest. Notice that the tails of our interest are as $x\downarrow 0$. For the sake of convenience we consider the transformed process
\[
\xi(t) = \eta(t)^{-3/2}, t\in\R,
\]
a semi-max-stable process with standard 1-Fr\'echet marginal distribution ($\proba(\xi(t)\le x) = \proba(\eta(t)\ge x^{-2/3}) = \exp(-x\inv)$). The bivariate tails of $\xi$ are asymptotically independent in the sense that the coefficient of residual tail dependence \citep{ledford96statistics} is 1/2. More precisely we have the following.
\begin{Lem}
 We have for all $s<t$, 
\[
\proba(\xi(s)>x, \xi(t)>x) \sim \frac4{3\pi}\frac1{(t-s)^3}\proba(\xi(s)>x)^2
\]
as $x\to\infty$. 
\end{Lem}
\begin{proof}
By stationarity, it suffices to consider $s=0$. By straightforward calculation,
\begin{align*}
\proba& (\xi(0)>x, \xi(t)>x)\\
& = 1-\proba(\eta(0)>x^{-2/3})-\proba(\eta(t)>x^{-2/3}) + \proba(\eta(0)>x^{-2/3},\eta(t)>x^{-2/3})\\
& = 1-2\exp(-x\inv)+\exp\pp{-\int_0^\infty \proba\pp{x^{2/3}\pp{\Z_0^w\wedge \Z_t^w}\le 1}\frac32 w^{1/2}dw}.
\end{align*}
Then, by self-similarity \eqref{eq:SS} of the tangent process, 
\begin{align*}
\int_0^\infty \proba\pp{x^{2/3}\pp{\Z_0^w\wedge \Z_t^w}\le 1}\frac32 w^{1/2}dw & =  x\inv \int_0^\infty \proba\pp{{\Z_0^w\wedge \Z_{tx^{1/3}}^w}\le 1}\frac32 w^{1/2}dw\\
& = x\inv \pp{1+\int_1^\infty\proba(\Z_{tx^{1/3}}^w\le 1)\frac 32w^{1/2}dw}\\
& = x\inv \pp{2-\int_0^1\proba(\Z_{tx^{1/3}}^w\le 1)\frac 32w^{1/2}dw},
\end{align*}
where in the last step we used the fact that $\Z$ has invariant distribution $w^{1/2}dw$,  $\int_0^\infty \proba(\Z_t^w\le 1)(3/2)w^{1/2}dw = \int_0^\infty \proba(\Z_0^w\le 1)(3/2) w^{1/2}dw = 1$. Then,
\begin{align*}
\int_0^1 & \proba(\Z_{tx^{1/3}}^w\le 1)\frac 32w^{1/2}dw \\
& = \int_0^1\int_0^1 \frac{2tx^{1/3}\sqrt y}{\pi[(y-w)^2+2(y+w)(tx^{1/3})^2 + (tx^{1/3})^4]}dy\frac32 w^{1/2}dw\\
& \sim \frac1{(tx^{1/3})^{3}}\int_0^1\int_0^1\frac{2\sqrt y}\pi dy \frac32w^{1/2}dw = \frac 4{3\pi t^3}\frac1x,
\end{align*}
as $x\to\infty$. Combining all the calculations, we have proved the lemma.
\end{proof}
\begin{Rem}The fact that the limit minimum process has asymptotically independent tails (near 0), intuitively, suggests that when the tangent process $\Z$ gets very close to the boundary at some time point, it drifts away within very short of time. This reflects actually similar behavior of the original $q$-Ornstein--Uhlenbeck process $X\topp q$ near the boundary ($b_q^-$ or $b_q^+$). 
\end{Rem}

\subsection*{Acknowledgements}The author thanks Wlodek Bryc, Sebastian Engelke, Zakhar Kabluchko, Stilian Stoev and Yimin Xiao for stimulating and inspiring discussions on the topic. The author thanks an anonymous referee for very careful reading of the manuscript and several suggestions, including in particular considering the residual tail dependence (Section \ref{sec:tails}). YW's research was partially supported by NSA grant  H98230-14-1-0318.
\bibliographystyle{apalike}
\bibliography{../include/references}

\end{document}

%% file: include_macros2014.tex
\def\cal{\mathcal}
\newcommand{\comment}[1]{}
\newcommand{\ind}{{\bf 1}}
\def\indd#1{{\ind}_{\{#1\}}}

\def\indn#1{\{#1_n\}_{n\in\N}}
\newcommand{\proba}{\mathbb P}
\newcommand{\esp}{{\mathbb E}}

\newcommand{\defe}{\mathrel{\mathop:}=}
\newcommand{\inv}{^{-1}}
\newcommand{\cov}{{\rm{Cov}}}

\newcommand{\calB}{{\cal B}}

\newcommand{\calD}{{\cal D}}

\def\B{{\mathbb B}}

\def\G{{\mathbb G}}

\def\Z{{\mathbb Z}}

\newcommand{\eqnh}{\begin{eqnarray*}}
\newcommand{\eqne}{\end{eqnarray*}}
\newcommand{\eqnhn}{\begin{eqnarray}}
\newcommand{\eqnen}{\end{eqnarray}}
\newcommand{\equh}{\begin{equation}}
\newcommand{\eque}{\end{equation}}

\def\summ#1#2#3{\sum_{#1 = #2}^{#3}}
\def\prodd#1#2#3{\prod_{#1 = #2}^{#3}}
\def\sif#1#2{\sum_{#1=#2}^\infty}

\newcommand{\eqd}{\stackrel{\rm d}{=}}

\def\topp#1{^{(#1)}}

\def\nn#1{{\left\|#1\right\|}}

\def\abs#1{\left|#1\right|}

\def\ccbb#1{\left\{#1\right\}}

\def\sccbb#1{\{#1\}}

\def\spp#1{(#1)}
\def\pp#1{\left(#1\right)} 
 
\def\bb#1{\left[#1\right]}

\def\mmid{\;\middle\vert\;}

\def\floor#1{\left\lfloor #1 \right\rfloor}
\def\sfloor#1{\lfloor #1 \rfloor}

\def\vv#1{{\boldsymbol #1}}

\def\qmand{\quad\mbox{ and }\quad}

\def\qmwith{\quad\mbox{ with }\quad}
\def\mfa{\mbox{ for all }}
\def\qmfa{\quad\mbox{ for all }\quad}
\def\mmas{\mbox{ as }}

\def\wt#1{\widetilde{#1}}

\def\what#1{\widehat{#1}}

\def\weakto{\Rightarrow}

\def\limn{\lim_{n\to\infty}}

\def\limsupn{\limsup_{n\to\infty}}

\def\Z{{\mathbb Z}}

\def\R{{\mathbb R}}

\def\N{{\mathbb N}}


%% file: qOU_EVT3.bbl
\def\cprime{$'$} \def\polhk#1{\setbox0=\hbox{#1}{\ooalign{\hidewidth
  \lower1.5ex\hbox{`}\hidewidth\crcr\unhbox0}}}
  \def\polhk#1{\setbox0=\hbox{#1}{\ooalign{\hidewidth
  \lower1.5ex\hbox{`}\hidewidth\crcr\unhbox0}}}
\begin{thebibliography}{}

\bibitem[Albin, 1990]{albin90extremal}
Albin, J. M.~P. (1990).
\newblock On extremal theory for stationary processes.
\newblock {\em Ann. Probab.}, 18(1):92--128.

\bibitem[Aldous, 1989]{aldous89probability}
Aldous, D. (1989).
\newblock {\em Probability approximations via the {P}oisson clumping
  heuristic}, volume~77 of {\em Applied Mathematical Sciences}.
\newblock Springer-Verlag, New York.

\bibitem[Biane, 1998]{biane98processes}
Biane, P. (1998).
\newblock Processes with free increments.
\newblock {\em Math. Z.}, 227(1):143--174.

\bibitem[Billingsley, 1999]{billingsley99convergence}
Billingsley, P. (1999).
\newblock {\em Convergence of probability measures}.
\newblock Wiley Series in Probability and Statistics: Probability and
  Statistics. John Wiley \& Sons Inc., New York, second edition.
\newblock A Wiley-Interscience Publication.

\bibitem[Blumenthal and Getoor, 1968]{blumenthal68markov}
Blumenthal, R.~M. and Getoor, R.~K. (1968).
\newblock {\em Markov processes and potential theory}.
\newblock Pure and Applied Mathematics, Vol. 29. Academic Press, New
  York-London.

\bibitem[Bo{\.z}ejko et~al., 1997]{bozejko97qGaussian}
Bo{\.z}ejko, M., K{\"u}mmerer, B., and Speicher, R. (1997).
\newblock {$q$}-{G}aussian processes: non-commutative and classical aspects.
\newblock {\em Comm. Math. Phys.}, 185(1):129--154.

\bibitem[Brown and Resnick, 1977]{brown77extreme}
Brown, B.~M. and Resnick, S.~I. (1977).
\newblock Extreme values of independent stochastic processes.
\newblock {\em J. Appl. Probability}, 14(4):732--739.

\bibitem[Brown, 1970]{brown70property}
Brown, M. (1970).
\newblock A property of {P}oisson processes and its application to macroscopic
  equilibrium of particle systems.
\newblock {\em Ann. Math. Statist.}, 41:1935--1941.

\bibitem[Bryc et~al., 2005]{bryc05probabilistic}
Bryc, W., Matysiak, W., and Szab{\l}owski, P.~J. (2005).
\newblock Probabilistic aspects of {A}l-{S}alam-{C}hihara polynomials.
\newblock {\em Proc. Amer. Math. Soc.}, 133(4):1127--1134 (electronic).

\bibitem[Bryc et~al., 2007]{bryc07quadratic}
Bryc, W., Matysiak, W., and Weso{\l}owski, J. (2007).
\newblock Quadratic harnesses, {$q$}-commutations, and orthogonal martingale
  polynomials.
\newblock {\em Trans. Amer. Math. Soc.}, 359(11):5449--5483.

\bibitem[Bryc and Wang, 2016]{bryc16local}
Bryc, W. and Wang, Y. (2016).
\newblock The local structure of $q$-{G}aussian processes.
\newblock {\em Probability and Mathematical Statistics}, 36(2):335--252.

\bibitem[Bryc and Weso{\l}owski, 2005]{bryc05conditional}
Bryc, W. and Weso{\l}owski, J. (2005).
\newblock Conditional moments of {$q$}-{M}eixner processes.
\newblock {\em Probab. Theory Related Fields}, 131(3):415--441.

\bibitem[Cheng and Xiao, 2016a]{cheng16excursion}
Cheng, D. and Xiao, Y. (2016a).
\newblock Excursion probability of {G}aussian random fields on sphere.
\newblock {\em Bernoulli}, 22(2):1113--1130.

\bibitem[Cheng and Xiao, 2016b]{cheng16mean}
Cheng, D. and Xiao, Y. (2016b).
\newblock The mean {E}uler characteristic and excursion probability of
  {G}aussian random fields with stationary increments.
\newblock {\em Ann. Appl. Probab.}, 26(2):722--759.

\bibitem[de~Haan, 1984]{dehaan84spectral}
de~Haan, L. (1984).
\newblock A spectral representation for max-stable processes.
\newblock {\em Ann. Probab.}, 12(4):1194--1204.

\bibitem[de~Haan and Ferreira, 2006]{dehaan06extreme}
de~Haan, L. and Ferreira, A. (2006).
\newblock {\em Extreme value theory}.
\newblock Springer Series in Operations Research and Financial Engineering.
  Springer, New York.
\newblock An introduction.

\bibitem[de~Haan and Pickands, 1986]{dehaan86stationary}
de~Haan, L. and Pickands, III, J. (1986).
\newblock Stationary min-stable stochastic processes.
\newblock {\em Probab. Theory Relat. Fields}, 72(4):477--492.

\bibitem[D{\c{e}}bicki et~al., 2016]{debicki16extremes}
D{\c{e}}bicki, K., Hashorva, E., and Ji, L. (2016).
\newblock Extremes of a class of nonhomogeneous {G}aussian random fields.
\newblock {\em Ann. Probab.}, 44(2):984--1012.

\bibitem[Dieker and Mikosch, 2015]{dieker15exact}
Dieker, A.~B. and Mikosch, T. (2015).
\newblock Exact simulation of {B}rown-{R}esnick random fields at a finite
  number of locations.
\newblock {\em Extremes}, 18(2):301--314.

\bibitem[Dombry and Eyi-Minko, 2013]{dombry13regular}
Dombry, C. and Eyi-Minko, F. (2013).
\newblock Regular conditional distributions of continuous max-infinitely
  divisible random fields.
\newblock {\em Electron. J. Probab}, 18(7):1--21.

\bibitem[Engelke and Kabluchko, 2015]{engelke15max}
Engelke, S. and Kabluchko, Z. (2015).
\newblock Max-stable processes and stationary systems of {L}\'evy particles.
\newblock {\em Stochastic Process. Appl.}, 125(11):4272--4299.

\bibitem[Engelke et~al., 2011]{engelke11equivalent}
Engelke, S., Kabluchko, Z., and Schlather, M. (2011).
\newblock An equivalent representation of the {B}rown-{R}esnick process.
\newblock {\em Statist. Probab. Lett.}, 81(8):1150--1154.

\bibitem[Engelke et~al., 2015]{engelke15maxima}
Engelke, S., Kabluchko, Z., and Schlather, M. (2015).
\newblock Maxima of independent, non-identically distributed {G}aussian
  vectors.
\newblock {\em Bernoulli}, 21(1):38--61.

\bibitem[Ethier and Kurtz, 1986]{ethier86markov}
Ethier, S.~N. and Kurtz, T.~G. (1986).
\newblock {\em Markov processes}.
\newblock Wiley Series in Probability and Mathematical Statistics: Probability
  and Mathematical Statistics. John Wiley \& Sons, Inc., New York.
\newblock Characterization and convergence.

\bibitem[Falconer, 2003]{falconer03local}
Falconer, K.~J. (2003).
\newblock The local structure of random processes.
\newblock {\em J. London Math. Soc. (2)}, 67(3):657--672.

\bibitem[Gin{\'e} et~al., 1990]{gine90max}
Gin{\'e}, E., Hahn, M.~G., and Vatan, P. (1990).
\newblock Max-infinitely divisible and max-stable sample continuous processes.
\newblock {\em Probab. Theory Related Fields}, 87(2):139--165.

\bibitem[Hashorva and Ji, 2016]{hashorva16extremes}
Hashorva, E. and Ji, L. (2016).
\newblock Extremes of {$\alpha(\mathbf{t})$}-locally stationary {G}aussian
  random fields.
\newblock {\em Trans. Amer. Math. Soc.}, 368(1):1--26.

\bibitem[Ismail, 2009]{ismail09classical}
Ismail, M. E.~H. (2009).
\newblock {\em Classical and quantum orthogonal polynomials in one variable},
  volume~98 of {\em Encyclopedia of Mathematics and its Applications}.
\newblock Cambridge University Press, Cambridge.

\bibitem[Kabluchko, 2009a]{kabluchko09extremes}
Kabluchko, Z. (2009a).
\newblock Extremes of space-time {G}aussian processes.
\newblock {\em Stochastic Process. Appl.}, 119(11):3962--3980.

\bibitem[Kabluchko, 2009b]{kabluchko09spectral}
Kabluchko, Z. (2009b).
\newblock Spectral representations of sum- and max-stable processes.
\newblock {\em Extremes}, 12(4):401--424.

\bibitem[Kabluchko and Schlather, 2010]{kabluchko10ergodic}
Kabluchko, Z. and Schlather, M. (2010).
\newblock Ergodic properties of max-infinitely divisible processes.
\newblock {\em Stochastic Process. Appl.}, 120(3):281--295.

\bibitem[Kabluchko et~al., 2009]{kabluchko09stationary}
Kabluchko, Z., Schlather, M., and de~Haan, L. (2009).
\newblock Stationary max-stable fields associated to negative definite
  functions.
\newblock {\em Ann. Probab.}, 37(5):2042--2065.

\bibitem[Kabluchko and Stoev, 2016]{kabluchko16stochastic}
Kabluchko, Z. and Stoev, S. (2016).
\newblock Stochastic integral representations and classification of sum- and
  max-infinitely divisible processes.
\newblock {\em Bernoulli}, 22(1):107--142.

\bibitem[Khoshnevisan, 1997]{khoshnevisan97escape}
Khoshnevisan, D. (1997).
\newblock Escape rates for {L}\'evy processes.
\newblock {\em Studia Sci. Math. Hungar.}, 33(1-3):177--183.

\bibitem[Leadbetter et~al., 1983]{leadbetter83extremes}
Leadbetter, M.~R., Lindgren, G., and Rootz{\'e}n, H. (1983).
\newblock {\em Extremes and related properties of random sequences and
  processes}.
\newblock Springer Series in Statistics. Springer-Verlag, New York.

\bibitem[Ledford and Tawn, 1996]{ledford96statistics}
Ledford, A.~W. and Tawn, J.~A. (1996).
\newblock Statistics for near independence in multivariate extreme values.
\newblock {\em Biometrika}, 83(1):169--187.

\bibitem[LePage et~al., 1981]{lepage81convergence}
LePage, R., Woodroofe, M., and Zinn, J. (1981).
\newblock Convergence to a stable distribution via order statistics.
\newblock {\em Ann. Probab.}, 9(4):624--632.

\bibitem[Oesting et~al., 2012]{oesting12simulation}
Oesting, M., Kabluchko, Z., and Schlather, M. (2012).
\newblock Simulation of {B}rown-{R}esnick processes.
\newblock {\em Extremes}, 15(1):89--107.

\bibitem[Penrose, 1991]{penrose91minima}
Penrose, M.~D. (1991).
\newblock Minima of independent {B}essel processes and of distances between
  {B}rownian particles.
\newblock {\em J. London Math. Soc. (2)}, 43(2):355--366.

\bibitem[Penrose, 1992]{penrose92semi}
Penrose, M.~D. (1992).
\newblock Semi-min-stable processes.
\newblock {\em Ann. Probab.}, 20(3):1450--1463.

\bibitem[Pickands, 1969]{pickands69asymptotic}
Pickands, III, J. (1969).
\newblock Asymptotic properties of the maximum in a stationary {G}aussian
  process.
\newblock {\em Trans. Amer. Math. Soc.}, 145:75--86.

\bibitem[Piterbarg, 1996]{piterbarg96asymptotic}
Piterbarg, V.~I. (1996).
\newblock {\em Asymptotic methods in the theory of {G}aussian processes and
  fields}, volume 148 of {\em Translations of Mathematical Monographs}.
\newblock American Mathematical Society, Providence, RI.
\newblock Translated from the Russian by V. V. Piterbarg, Revised by the
  author.

\bibitem[Resnick, 1987]{resnick87extreme}
Resnick, S.~I. (1987).
\newblock {\em Extreme values, regular variation, and point processes},
  volume~4 of {\em Applied Probability. A Series of the Applied Probability
  Trust}.
\newblock Springer-Verlag, New York.

\bibitem[Revuz and Yor, 1999]{revuz99continuous}
Revuz, D. and Yor, M. (1999).
\newblock {\em Continuous martingales and {B}rownian motion}, volume 293 of
  {\em Grundlehren der Mathematischen Wissenschaften [Fundamental Principles of
  Mathematical Sciences]}.
\newblock Springer-Verlag, Berlin, third edition.

\bibitem[Stoev, 2008]{stoev08ergodicity}
Stoev, S.~A. (2008).
\newblock On the ergodicity and mixing of max-stable processes.
\newblock {\em Stochastic Process. Appl.}, 118(9):1679--1705.

\bibitem[Stoev and Taqqu, 2005]{stoev06extremal}
Stoev, S.~A. and Taqqu, M.~S. (2005).
\newblock Extremal stochastic integrals: a parallel between max-stable
  processes and {$\alpha$}-stable processes.
\newblock {\em Extremes}, 8(4):237--266 (2006).

\bibitem[Szab{\l}owski, 2012]{szablowski12qWiener}
Szab{\l}owski, P.~J. (2012).
\newblock $q$-{W}iener and ($\alpha$,$q$)-{O}rnstein--{U}hlenbeck processes. a
  generalization of known processes.
\newblock {\em Theory of Probability \& Its Applications}, 56(4):634--659.

\bibitem[Wang, 2016]{wang16large}
Wang, Y. (2016).
\newblock Large jumps of {$q$}-{O}rnstein--{U}hlenbeck processes.
\newblock {\em Statist. Probab. Lett.}, 118:110--116.

\bibitem[Weintraub, 1991]{weintraub91sample}
Weintraub, K.~S. (1991).
\newblock Sample and ergodic properties of some min-stable processes.
\newblock {\em Ann. Probab.}, 19(2):706--723.

\bibitem[Xiao, 1998]{xiao98asymptotic}
Xiao, Y. (1998).
\newblock Asymptotic results for self-similar {M}arkov processes.
\newblock In {\em Asymptotic methods in probability and statistics ({O}ttawa,
  {ON}, 1997)}, pages 323--340. North-Holland, Amsterdam.

\end{thebibliography}
